\documentclass[oneside,english,12pt,reqno]{amsart}
\usepackage[colorlinks=true,linkcolor=blue,citecolor=magenta]{hyperref}
\usepackage{comment}
\usepackage[pagewise]{lineno}\linenumbers

\allowdisplaybreaks

\usepackage[section]{placeins}

\usepackage[mathscr]{eucal}
\usepackage{typearea}
\usepackage{charter}
\usepackage{graphicx}

\usepackage{units}
\usepackage{mathrsfs}
\usepackage{amstext}
\usepackage{amsthm}
\usepackage{esint}
\usepackage{amsmath}
\usepackage{amssymb}
\usepackage{soul}
\usepackage{cancel}

\usepackage{a4wide}
\usepackage{enumitem}
\newcommand{\bra}[1]{\left(#1\right)}

\renewcommand{\L}{\mathscr{L}}
\renewcommand{\H}{\mathscr{H}}

\makeatletter
\numberwithin{equation}{section}
\numberwithin{figure}{section}
\newtheorem{corollary}{Corollary}[section]
\newtheorem{theorem}{Theorem}[section]
\newtheorem{lemma}{Lemma}[section]
\newtheorem{remark}{Remark}[section]

\makeatother

\usepackage{babel}

\begin{document}
\title[]{Global Well-Posedness and Asymptotic Behavior for a Reaction-Diffusion System of Competition Type}
\author{Jeffrey Morgan, Samia Zermani}

\address{Jeffrey Morgan \\Department of Mathematics,
University of Houston, Houston, TX 77204-3008, USA}
\email{jjmorgan@uh.edu}

\address{Samia Zermani \\Tunis El Manar University,
Institut Pr\'eparatoire aux Etudes d'ing\'enieurs El Manar,   Campus
Universitaire Farhat Hached, 2092 Tunis, Tunisia}
\email{samia.zermani@ipeiem.utm.tn}
\begin{abstract}
We analyze a reaction-diffusion system describing the growth of microbial species in a model of flocculation type that arises in biology. Existence of global classical positive solutions is proved under general growth assumptions, with flocculation and deflocculation rates polynomially bounded above, that guarantee uniform sup norm
bounds for all time t  obtained by an $L^{p}-$energy functional estimate. We also show finite time blow up can occur when the yield coefficients are large enough. Also, using arguments relying on the spectral and fixed theory, we show persistence and existence of nonhomogenous population steady-states. Finally, we present some numerical simulations to show the combined effects of motility coefficients and the flocculation-deflocculation rates on the coexistence of species.
\end{abstract}

\maketitle
\section{Introduction}
This paper studies a reaction-diffusion system modeling a flocculation process in an un-stirred chemostat, where the isolated or planktonic bacteria naturally aggregate, reversibly, to one another to form macroscopic flocs. Our work is an extension of the several species flocculation models in \cite{FH} and \cite{FH2}, and the single species model considered in \cite{ZA}.

We start by reviewing the model in \cite{ZA} and relating the results to our current work. The model considered in that work is given by
\begin{equation}\label{Sa1}
\left\{
\begin{array}{cc}
S_{t}=d_{0}S_{xx}-S_{x}-f(S)u-g(S)v,&\text{ on }(0,1)\times(0,T),\\
{u}_{t}=d_1{u}_{xx}-{u}_{x}+f(S)u-\dfrac{1}{y_{u}}\alpha(u,v)u+\beta(u,v)v,&\text{ on }(0,1)\times(0,T),\\
{v_i}_{t}=d_2{v}_{xx}-{v}_{x}+g(S)v+\alpha(u,v)u-\dfrac{1}{y_{v}}\beta(u,v)v,&\text{ on }(0,1)\times(0,T),\\
\end{array}%
\right.
\end{equation}%
with boundary conditions
\begin{equation}\label{Sa2}
\left\{
\begin{array}{cc}
-d_{0}S_{x}(0,t)+S(0,t)=1,\,\,S_{x}(1,t)=0,&0<t<T,\\
-d_{1}{u}_{x}(0,t)+u(0,t)=0,\,\,{u_i}_{x}(1,t)=0,&0<t<T,\\
-d_{2}{v}_{x}(0,t)+v(0,t)=0,\,\,{v_i}_{x}(1,t)=0,&0<t<T,\\
\end{array}%
\right.
\end{equation}%
and initial conditions
\begin{equation}\label{Sa3}
\left\{\begin{array}{cc}
S(x,0)={S_0}(x),\,\, u(x,0)={u_0}(x)=,\,\,v(x,0)={v_0}(x)=,&x\in[0,1]\\
i=1,...,m&
\end{array}\right.
\end{equation}%
Here, $d_0,d_1,d_2,y_u,y_v>0$, $f,g\in C^1(\mathbb{R}_+,\mathbb{R}_+)$ satisfy $f(0)=g(0)=0$, $\alpha,\beta\in C^1(\mathbb{R}_+^2,\mathbb{R}_+)$, and $S_0,u_0,v_0\in C([0,1],\mathbb{R}_+)$, where $\mathbb{R}_+=[0,\infty)$. It is well known that these conditions guarantee the existence of a unique classical componentwise nonnegative solution $(S,u,v)$ to (\ref{Sa1})-(\ref{Sa3}) on a maximal time interval $(0,T_{\text{max}})$. Furthermore, if $T_{\text{max}}<\infty$ then
$$\limsup_{t\to T_{\text{max}}^-}\left(\|S(\cdot,t)\|_{\infty,(0,1)}+\|u(\cdot,t)\|_{\infty,(0,1)}+\|v(\cdot,t)\|_{\infty,(0,1)}\right)=\infty$$
The authors of \cite{ZA} established a few different global existence results. The first of these assumes the conditions above, along with
\begin{enumerate}
\item[$\mathbf{(ZA_{1})}$] $y_{u}y_{v}<1$,
\item[$\mathbf{(ZA_2)}$] $\alpha(u,v),\beta(u,v)$ are nondecreasing in $u,v\ge0$,
\item[$\mathbf{(ZA_3)}$] there exists a constant $k>0$ so that
$$\lim_{z\to\infty}\left(\frac{1}{y_u}\alpha(z,0)-k\beta(z,kz)\right)=\infty,$$
$$\lim_{z\to\infty}\left(\frac{1}{y_v}\beta(0,kz)-k\alpha(z,kz)\right)=\infty.$$
\end{enumerate}
The condition in $\mathbf{(ZA_3)}$ places considerable further restrictions on the values $y_u$ and $y_v$ even when $\alpha(u,v)$ and $\beta(u,v)$ are linear in $u$ and $v$. For example, if $a,b,c,d\ge0$ (recall that $\alpha(u,v),\beta(u,v)\ge0$ for $u,v\ge0$) and
$$\alpha(u,v)=au+bv,\,\,\beta(u,v)=cu+dv,$$
then $\mathbf{(ZA_3)}$ implies there exists $k>0$ so that
$$\frac{1}{y_u}a-k(c+kd)>0$$
and
 $$\frac{1}{y_v}d-(a+kb)>0.$$
 The first inequality implies
 $$a>y_uk(c+kd),$$
 and the second inequality implies
 $$a<\frac{d}{y_v}-kb.$$
 By contrast, our results imply that in the case when $\alpha(u,v)$ and $\beta(u,v)$ are nonnegative and linear in $u\ge0$ and $v\ge0$, the condition $y_uy_v\le1$ is sufficient to obtain global existence for (\ref{Sa1})-(\ref{Sa3}). More generally, we can obtain global existence when $\alpha(u,v)$ and $\beta(u,v)$ are somewhat nonlinear, provided $y_uy_v\le 1$, $\alpha(u,v),\beta(u,v)\le C(u+v+1)^l$ for some $l>0$ (with no restriction on the size of $l$), and there exists $K>0$ and $0<\varepsilon<<1$ so that either
$$\alpha(u,v)\le K(u+v+1)^{2-\varepsilon}$$
or
$$\beta(u,v)\le K(u+v+1)^{2-\varepsilon}.$$

 The second global existenced result for (\ref{Sa1})-(\ref{Sa3}) in \cite{ZA} assumes $\alpha$ and $\beta$ are uniformly bounded, along with a condition that relates the growth of $f(S)$ and $g(S)$ as $S\to\infty$ to $y_u$ and $y_v$. In this case, regardless of the choices of $y_u,y_v>0$, they prove global existence. However, in this case, there really is no need for bounds on the growth of $f(S)$ and $g(S)$ since $S$ can be a priori sup norm bounded, and then the reaction terms in the equations for $u$ and $v$ are bounded above by linear expressions in $u$ and $v$, making it a simple matter to obtain global existence.

 The final global existence result for (\ref{Sa1})-(\ref{Sa3}) in \cite{ZA} only assumes $0< y_u,y_v\le1$ and $d_0=d_1=d_2>0$. However, in this case, if we define
 $$w=S+u+v,$$
then
 \begin{equation*}\left\{
 \begin{array}{cc}
 w_t\le d_0w_{xx}-w_x,&\text{on }(0,1)\times(0,T),\\
 -d_0w_x(0,t)+w(0,t)=1,&\text{on }0<t<T,\\
 w(x,0)=u_0(x)+v_0(x)+w_0(x),&x\in(0,1).
\end{array}\right.
\end{equation*}
As a result, the comparison principle implies $w$ is sup norm bounded, and consequently, $S$, $u$ and $v$ are sup norm bounded. This immediately gives global existence.

\medskip Our multi component generalization of (\ref{Sa1})-(\ref{Sa3}) is motivated by the several species models in \cite{FH}-\cite{FH2}. There, the authors studied a chemostat model where $$\alpha_i(u,v)=a_i\sum_{i=1}^m(u_i+v_i)$$ and $\beta_i(\cdot)=b_i$, with $a_i,b_i>0$ for all $i=1,...,m$. Their model was only driven by reaction terms representing instantaneous interactions in which spatial effects were considered negligible due to concentrations being homogeneously distributed, and was given by
\begin{equation*}
\left\{
\begin{array}{l}
S_{t}=D(S_{in}-S)-\sum_{i=1}^m\left(f_i(S)u_i+g_i(S)v_i\right),\\
{u_i}_{t}=(f_i(S)-D_{0,i})u_i-\alpha_i(\cdot)u_i+\beta_i(\cdot)v_i,\quad i=1,...,m,\\
{v_i}_{t}=(g_i(S)-D_{1,i})v_i+\alpha_i(\cdot)u_i-\beta_i(\cdot)v_i,\quad i=1,...,m.\\
\end{array}%
\right.
\end{equation*}
This system, and the work in \cite{ZA} has motivated us to consider the system
\begin{equation}\label{S1}
\left\{
\begin{array}{cc}
S_{t}=d_{0}S_{xx}-S_{x}-\sum_{i=1}^m\left(f_i(S)u_i+g_i(S)v_i\right),&\text{ on }(0,1)\times(0,T),\\
{u_i}_{t}=d_{u,i}{u_i}_{xx}-{u_i}_{x}+f_i(S)u_i-\dfrac{1}{y_{u,i}}\alpha_i(u,v)u_i+\beta_i(u,v)v_i,&\text{ on }(0,1)\times(0,T),\\i=1,...,m,&\\
{v_i}_{t}=d_{v,i}{v_i}_{xx}-{v_i}_{x}+g_i(S)v_i+\alpha_i(u,v)u_i-\dfrac{1}{y_{v,i}}\beta_i(u,v)v_i,&\text{ on }(0,1)\times(0,T),\\i=1,...,m,&\\
\end{array}%
\right.
\end{equation}%
with boundary conditions:
\begin{equation}\label{S2}
\left\{
\begin{array}{cc}
-d_{0}S_{x}(0,t)+S(0,t)=\gamma_{S},\,\,S_{x}(1,t)=0,&t>0\\
-d_{u,i}{u_i}_{x}(0,t)+u_i(0,t)=\gamma_{u,i},\,\,{u_i}_{x}(1,t)=0,&t>0,i=1,...,m,\\
-d_{v,i}{v_i}_{x}(0,t)+v_i(0,t)=\gamma_{v,i},\,\,{v_i}_{x}(1,t)=0,&t>0,i=1,...,m,\\
\end{array}%
\right.
\end{equation}%
and initial conditions:
\begin{equation}\label{S3}
\left\{\begin{array}{cc}
S(x,0)={S_0}(x),\,\, u(x,0)={u_0}(x)=\left({u_0}_i(x)\right),\,\,v(x,0)={v_0}(x)=\left({v_0}_i(x)\right),&x\in[0,1]\\
i=1,...,m&
\end{array}\right.
\end{equation}%
Here $S(t)$ is the substrate concentration, $u(\cdot,t)=\left(u_i(\cdot,t)\right)$ and $v(\cdot,t)=\left(v_i(\cdot,t)\right)$ denote respectively the concentrations of isolated and attached bacteria at time $t$, $f=\left(f_i\right)$ and $g=\left(g_i\right)$ represent, respectively, the per-capita growth rates of the isolated and attached bacteria, $\alpha=\left(\alpha_i\right)$ and $\beta=\left(\beta_i\right)$ denote, respectively, the componentwise nonnegative flocculation and deflocculation rates. The coefficients $y _{u,i}$ and $y _{v,i}$ are positive constants that respectively consider the characteristics of the medium, the efficiency of collision and  the yield coefficient for free and attached bacteria. The isolated and attached  microbial cells are assumed to be capable of random movement, modeled by diffusion with diffusivity constants $d_0,d_{u,i},d_{v,i}>0$, and the constants $\gamma_{S},\gamma_{u,i},\gamma_{v,i}\ge0$ represent feed terms at one end of the interval. Finally, each of $S_0$, the components ${u_0}_i$ of $u_0$, and ${v_0}_i$ of $v_0$, represent bounded nonnegative initial data.

Our results guarantee the existence of global componentwise nonnegative \textbf{weak} solutions of (\ref{S1})-(\ref{S3}) when ${y_u}_i{y_v}_i\le 1$ for $i=1,...,m$, and global well-posedness of componentwise nonnegative classical solutions to (\ref{S1})-(\ref{S3}) when ${y_u}_i{y_v}_i\le 1$ for $i=1,...,m$, and reasonable growth is assumed on the components of $\alpha(u,v)$ and $\beta(u,v)$. We also show that finite time blow up can occur under these same reasonable growth assumptions when ${y_u}_i{y_v}_i> 1$ for some $i\in\{1,...,m\}$. Finally, we examine the existence of steady states in the case of (\ref{Sa1})-(\ref{Sa2}) under various conditions.

Section 2 introduces basic notation, and states and proves our result related to global weak solutions. Section 3 states our results for global existence of global classical solutions, and blow-up. Section 4 contains the proofs of these results, and Section 5 gives steady state results. In our study, we extend the model in \cite{ZA} by dropping the assumption that flocculation and deflocculation rates are bounded functions. Indeed, when the washout steady-state $(S,u,v)=(1,0,0)$ is unstable, the authors, in \cite{ZA}, show the existence of a non trivial steady-state solution under the assumption that the flocculation-deflocculation rates $\alpha(\cdot),\,\beta(\cdot)$ are bounded functions with $y_{u},y_{v}$ satisfying that $\displaystyle \frac{y_{u}f(1)+1}{1+y_{u}\lambda_{d_{1}}}=\frac{y_{v}g(1)+1}{1+y_{v}\lambda_{d_{2}}}$, noting that the study of the case of the extinction of one of the species in \cite{ZA} is shown only by numerical simulations.\\ In this work, for a different range of parameters than in \cite{ZA}, we prove coexistence of species as well as the extinction of one of the species in the medium. Finally, in Section 6, we provide some numerical simulations
for different values of the diffusion coefficients to show that the competition
between species does not depend only on the growth rates but also on the diffusion coefficients.
\medskip

\section{Local Well-Posedness and Global Weak Solutions}
Throughout this work, we denote $\mathbb{R}_+=[0,\infty)$, and we make the following assumptions concerning the terms in (\ref{S1})-(\ref{S3}).
 \medskip
 \begin{enumerate}
\item[\textbf{(A1)}] $y_{u,i},y_{v,i},d_0,d_{u,i},d_{v,i}>0$ and $\gamma_{S},\gamma_{u,i},\gamma_{v,i}\ge0$ for $i=1,...,m$.
\item[\textbf{(A2)}] $f,g:\mathbb{R}_+\to\mathbb{R}_+^m$ and $\alpha,\beta:\mathbb{R}_+^m\times\mathbb{R}_+^m\to\mathbb{R}_+^m$ are locally Lipschitz, and $f(0)=g(0)=0$ for all $u,v\in\mathbb{R}_+^m$.
\item[\textbf{(A3)}] $S_0,{u_0}_i,{v_0}_i\in C([0,1],\mathbb{R}_+)$ for $i=1,...,m$.
\end{enumerate}
\medskip
\noindent We show below that if (A1)-(A3) and $y_{u,i}y_{v,i}<1$ for all $i=1,...,m$, then (\ref{S1})-(\ref{S3}) has a global componentwise nonnegative weak solution.
\medskip
Our discussion of the existence of global componentwise nonnegative weak solutions to (\ref{S1})-(\ref{S2}) takes advantage of the work in \cite{Pie1,Pie2} regarding the existence of global componentwise nonnegative weak solutions to reaction-diffusion systems whose reaction vector fields satisfy assumptions associated with quasi-positivity and dissipation of mass. In order to make use of their results, we create a general reaction-diffusion system that contains (\ref{S1})-(\ref{S3}) as a special case. To this end, suppose $\tilde{m}\in\mathbb{N}$ and consider the system given by
\begin{equation}\label{generalrd}\left\{
\begin{array}{cc}
{w_i}_t=D_i {w_i}_{xx}-{w_i}_x+F_i(x,t,w),&x\in(0,1),0<t<T,i=1,...,\tilde{m},\\
-D_i{w_x}_i(0,t)+w_i(0,t)=\delta_i,\,{w_i}_x(1,t)=0&0<t<T,i=1,...,\tilde{m},\\
w_i(x,0)={w_0}_i(x),&x\in(0,1),i=1,...,\tilde{m}.
\end{array}\right.
\end{equation}
Here, we assume
\medskip
\begin{enumerate}
\item[\textbf{(A4)}] for each $i=1,...,\tilde{m}$, $D_i>0$, $\delta_i\ge0$, $F_i:(0,1)\times\mathbb{R}_+\times\mathbb{R}_+^{\tilde{m}}\to\mathbb{R}_+^{\tilde{m}}$ is bounded on $(0,1)\times(0,T)\times(0,M)$ for each $T,M>0$, $F_i(x,t,w)$ is locally Lipschitz in $w$, uniformly with respect to $x$ and bounded $t$, and ${w_0}_i\in C([0,1],\mathbb{R}_+)$.
\end{enumerate}
\medskip
\noindent It should be clear that (\ref{S1})-(\ref{S2}) is a special case of (\ref{generalrd}).

We say that $F=(F_i)$ satisfies a quasi-positivity condition if and only if
\medskip
\begin{enumerate}
\item[\textbf{(QP)}] for each $i=1,...,\tilde{m}$, $F_i(x,t,w)\ge0$ for all $0<x<1$, $t>0$ and $w\in\mathbb{R}_+^{\tilde{m}}$ with $w_i=0$,
\end{enumerate}
\medskip
\noindent and $F$ satisfies a dissipation of mass condition if and only if there are scalars $a_i>0$ for $i=1,...,\tilde{m}$, $K_1\in\mathbb{R}$ and $K_2\ge0$ so that
\medskip
\begin{enumerate}
\item[\textbf{(DISS)}] $\sum_{i=1}^{\tilde{m}}a_iF_i(x,t,w)\le K_1\sum_{i=1}^{\tilde{m}}w_i+K_2.$
\end{enumerate}
\medskip
\noindent It is well known that (A4) and (QP) guarantee the existence of a unique componentwise nonnegative classical solution $w$ to (\ref{generalrd}) on a maximal time interval $(0,T_{\text{max}})$, and furthermore,
$$\limsup_{t\to T_{\text{max}}}\sum_{i=1}^{\tilde{m}}\|w_i(\cdot,t)\|_{\infty,(0,1)}=\infty.$$
The local well-posedness is guaranteed by a wealth of results, including \cite{A1,A2}, and the componentwise nonnegativity is guaranteed by results in \cite{Kui} and many others. We state this result below.

\begin{theorem}\label{local}
If (A4) and (QP) are satisfied, then there exists $T_\text{max}\in(0,\infty]$ such that (\ref{generalrd}) has a unique componentwise nonnegative classical solution on the maximal interval $(0,T_\text{max})$. Furthermore, if $T_\text{max}<\infty$ then
$$
\limsup_{t\to T_{\text{max}}^-}\sum_{i=1}^m\|w_i(\cdot,t)\|_{\infty,(0,1)}=\infty.
$$
\end{theorem}

\medskip\noindent The solution guaranteed by Theorem \ref{local} is said to be a \textbf{global classical solution} provided $T_\text{max}=\infty$.

The system (\ref{generalrd}) is closely related to the system
\begin{equation}\label{generalrdtr}
\left\{
\begin{array}{cc}
{w_i}_t=D_i {w_i}_{xx}-{w_i}_x+\phi(w)F_i(x,t,w),&x\in(0,1),0<t<T,i=1,...,\tilde{m},\\
-D_i{w_x}_i(0,t)+w_i(0,t)=\delta_i,\,{w_i}_x(1,t)=0&0<t<T,i=1,...,\tilde{m},\\
w_i(x,0)={w_0}_i(x),&x\in(0,1),i=1,...,\tilde{m},
\end{array}
\right.
\end{equation}
where $\phi\in C_0^\infty\left(\mathbb{R}_+^{\tilde{m}},[0,1]\right)$. We remark that if (A4) and (QP) are satisfied, then Theorem \ref{local} immediately applies to (\ref{generalrdtr}), and since the functions $\phi(w)F_i(x,t,w)$ are Lipschitz in $w$, uniformly for $(x,t)\in (0,1)\times(0,T)$ for each $T>0$ and $i=1,...,\tilde{m}$, solutions to (\ref{generalrdtr}) cannot blow-up in finite time. Consequently, solutions to (\ref{generalrdtr}) are global. We state this result along with another critical piece below.

\begin{corollary}\label{Cor1}
If (A4) and (QP) are satisfied, and $\phi\in C_0^\infty\left(\mathbb{R}_+^{\tilde{m}},[0,1]\right)$, then (\ref{generalrdtr}) has a unique componentwise nonnegative classical global solution. Furthermore, if there exists $M\in C(\mathbb{R}_+,\mathbb{R}_+)$ independent of $\phi$ so that the solution $w$ of (\ref{generalrdtr}) satisfies $\|w_i(\cdot,t)\|_{\infty,(0,1)}\le M(t)$ for all $i=1,...,m$ and $t\ge 0$, then (\ref{generalrd}) has a unique componentwise nonnegative classical global solution.
\end{corollary}

Note that if (A4), (QP) and (DISS) are satisfied, and we integrate the first equation in (\ref{generalrdtr}) over $(0,1)$ and apply (DISS), then
$$\frac{d}{dt}\int_0^1\sum_{i=1}^{\tilde{m}}a_iw_i(x,t)dx\le \sum_{i=1}^{\tilde{m}}a_i\delta_i+\int_0^1\left( K_1\sum_{i=1}^{\tilde{m}}w_i(x,t)+K_2\right)dx,$$
for all $t>0$, independent of $\phi$. This implies that for each $i=1,...,\tilde{m}$, $\|w_i(\cdot,t)\|_{1,(0,1)}$ grows at most exponentially as $t\to\infty$ if $K_1>0$, linearly as $t\to\infty$ if $K_1=0$, and is uniformly bounded for all $t\ge0$ if either $K_1<0$ or $K_1,K_2\le0$. As a result, solutions to (\ref{generalrd}) satisfy an $L^1(0,1)$ a priori bound. It is also possible to modify results in \cite{morgan1989global,Mor90} to show these solutions satisfy an $L^2((0,1)\times(0,T))$ a priori bound for all $T>0$. Unfortunately, it is well known that this is not sufficient to guarantee global existence of solutions to (\ref{generalrd}). In this regard, we encourage the reader to explore \cite{Pie3} for a discussion of blow-up when (A4), (QP) and (DISS) are satisfied.

\medskip Our interest in the remainder of this section is global existence of componentwise nonnegative weak solutions to (\ref{S1})-(\ref{S2}). In this regard, we refer the reader to the work in \cite{Pie1,Pie2}. Although their presentation is made (in arbitrary dimensions) for reaction-diffusion systems satisfying (QP) and (DISS), without the inclusion of first order spatial derivative terms, and with homogeneous boundary conditions, it is clear that the inclusion of the first order derivative terms in (\ref{generalrd}) and our choice of nonhomogeneous boundary conditions have no effect on the analysis in their work.  Following their approach, we define $Q_T=(0,1)\times(0,T)$,
$$\mathbb{D}_T=\left\{\psi\in C^\infty\left(\overline{Q}_T\right);\, \psi(\cdot,T)=0\right\},$$
and we say that $w=(w_i)$ is a componentwise nonnegative global weak solution to (\ref{generalrd}) provided
\medskip
\begin{equation}\label{weak}
\left\{\begin{array}{c}
\text{for each }i=1,...,\tilde{m}\text{ and }T>0,\, F_i(\cdot,\cdot,w)\in L^1(Q_T),\\ w_i, {w_i}_x\in L^1(Q_T),\,w_i\ge0\text{ a.e., and for all }\psi\in \mathbb{D}_T,\\
-\int_0^1\psi(x,0) {w_0}_i(x)dx+\int_0^T\int_0^1\left(-\psi_tw_i+D_i\psi_x{w_i}_x\right)dxdt+\int_0^T\psi(1,t)w_i(1,t)dt\\=\int_0^T\int_0^1\psi F_i(x,t,w)dxdt+\delta_i\int_0^T\psi(0,t)dt.
\end{array}\right.
\end{equation}

\medskip
\noindent From \cite{Pie1,Pie2}, it is possible to prove the existence of a global weak solution to  (\ref{generalrd}) provided (A4) and (QP) are true, and for every $T>0$ there exists $C_T>0$ independent of $\phi\in C_0^\infty\left(\mathbb{R}_+^{\tilde{m}},[0,1]\right)$, so that if $w$ is the unique componentwise nonnegative classical global solution to (\ref{generalrdtr}) then
\begin{equation}\label{L1estimate}
\int_0^T\int_0^1\left|\phi(w)F_i(x,t,w)\right|dxdt\le C_T\text{ for all }i=1,...,\tilde{m}.
\end{equation}
This result is accomplished by creating a sequence of cut-off functions $\phi_k\in C_0^\infty\left(\mathbb{R}_+^{\tilde{m}},[0,1]\right)$ such that $\phi_k(w)=1$ when $\|w\|\le k$, and then showing that the solutions $w^k$ to (\ref{generalrdtr}) with $\phi=\phi_k$ has a subsequence that converges to a weak solution of (\ref{generalrd}) as $k\to\infty$.

\medskip We state this result below, without proof.

\medskip
\begin{theorem}\label{thm1}
If (A4) and (QP) are satisfied, and for every $T>0$ there exists $C_T>0$ independent of $\phi\in C_0^\infty\left(\mathbb{R}_+^{\tilde{m}},[0,1]\right)$, so that if $w$ is the unique componentwise nonnegative classical global solution to (\ref{generalrdtr}) then (\ref{L1estimate}) is true,
then (\ref{generalrd}) has a componentwise nonnegative global weak solution.
\end{theorem}
\medskip
Now we are ready to prove our first result related to the system (\ref{S1})-(\ref{S3}).
\medskip
\begin{theorem}\label{thm2}
Assume (A1)-(A3) and $y_{u,i}y_{v,i}<1$ for all $i=1,...,m$. Then (\ref{S1})-(\ref{S3}) has a componentwise nonnegative global weak solution.
\end{theorem}
\medskip
\begin{proof}
We start by reordering the equations in (\ref{S1})-(\ref{S2}) and examine it in the form (\ref{generalrd}). To this end, we set $\tilde{m}=2m+1$ and define
$$w=(w_1,w_2,w_3,...,w_{2m},w_{2m+1})=(S,u_1,v_1,u_2,v_2,....,u_m,v_m),$$
$$(D_1,D_2,D_3,...,D_{2m},D_{2m+1})=(d_0,d_{u,1},d_{v,1},...,d_{u,m},d_{v,m}),$$
$F(x,t,w)=(F_i(x,t,w))$ given by

\begin{equation*}
\left(
\begin{array}{c}
F_1(x,t,w)=-\sum_{i=1}^m\left(f_i(S)+g_i(S)\right)\\
F_2(x,t,w)=f_1(S)u_1-\dfrac{1}{y_{u,1}}\alpha_1(u,v)u_i+\beta_1(u,v)v_1\\
F_3(x,t,w)=g_1(S)v_1+\alpha_1(u,v)u_1-\dfrac{1}{y_{v,1}}\beta_1(u,v)v_1\\
\vdots\\
F_{2m}(x,t,w)=f_m(S)u_m-\dfrac{1}{y_{u,m}}\alpha_m(u,v)u_m+\beta_m(u,v)v_m\\
F_{2m+1}(x,t,w)=g_m(S)v_m+\alpha_m(u,v)u_m-\dfrac{1}{y_{v,m}}\beta_m(u,v)v_m\\
\end{array}
\right),
\end{equation*}
\smallskip
\begin{equation*}
(\delta_1,\delta_2,\delta_3,...,\delta_{2m},\delta_{2m+1})=(\gamma_S,\gamma_{u,1},\gamma_{v_1},...,\gamma_{u,2m},\gamma_{v,2m}),
\end{equation*}
and
\begin{equation*}
w_0(x)=(S_0(x),{u_0}_1(x),{v_0}_1(x),...,{u_0}_{2m}(x),{v_0}_{2m}(x)).
\end{equation*}
Conditions (A1)-(A3) guarantee (A4) and (QP). Now, suppose $\phi\in C_0^\infty\left(\mathbb{R}_+^{\tilde{m}},[0,1]\right)$ and $w$ isi the componentwise nonnegative global classical solution to (\ref{generalrdtr}). The maximum principle implies $w_1=S$ is sup norm bounded independent of $\phi$. In addition, integrating the equation for $w_1=S$ on $(0,1)\times(0,T)$ results in
\begin{align}\label{figiL1}
\int_0^T\int_0^1\left(\left|\phi f_i(S)u_i\right|+\left|\phi g_i(S)v_i\right|\right)dxdt&=\int_0^T\int_0^1\left(\phi f_i(w_1)w_{2i}+\phi g_i(w_1)w_{2i+1}\right)dxdt\nonumber\\&\le \gamma_ST+\|S_0\|_{\infty,(0,1)}
\end{align}
for each $i=1,...,m$. In addition, by integrating $$(y_{u,i}+1)w_{2i}+(y_{v,i}+1)w_{2i+1}=(y_{u,i}+1)u_i+(y_{v,i}+1)v_i,$$ we get
\begin{align}\label{alphabetaL1}
\int_0^T\int_0^1\left(\phi \alpha_i(w)w_{2i}+\phi \beta_i(w)w_{2i+1}\right)dxdt
\le &\delta\left(\left(\gamma_{u_i}+\gamma_{v_i}+\gamma_S\right)T\right.\nonumber\\&\left.+\|{u_0}_i+{v_0}_i\|_{\infty,(0,1)}+\|S_0\|_{\infty,(0,1)}\right)
\end{align}
for each $i=1,...,m$, where $$\delta>2\max_{j=1,...,m}\left\{\frac{y_{v,j}(1+y_{u,j})}{1-y_{u,j}y_{v,j}},\frac{y_{u,j}(1+y_{v,j})}{1-y_{u,j}y_{v,j}}\right\}.$$
As a result, \ref{figiL1}) and (\ref{alphabetaL1}) imply (\ref{L1estimate}) is true. Therefore, Theorem \ref{thm1} implies (\ref{generalrd}) has a componentwise nonnegative global weak solution.
\end{proof}

\medskip We remark that if $T_{\text{max}}<\infty$ in Theorem \ref{local}, then the global weak solution referred to in Theorem \ref{thm2} is not guaranteed to be unique for $t\ge T_{\text{max}}$. We refer the interested reader to \cite{Pie1,Pie2} for additional comments on this subject.

\medskip

\section{Global Classical Solutions for (\ref{S1})-(\ref{S3})}
It does not appear possible to prove that the hypotheses in Theorem \ref{thm2} imply (\ref{S1})-(\ref{S3}) has a unique componentwise nonnegative global classical solution. However, this can be achieved if we impose the following reasonable growth conditions, even if we only assume $y_{u,i}y_{v,i}\le1$ for all $i=1,...,m$.
\begin{enumerate}
\item[\textbf{(A5)}] There exist $l\ge 1$ and $h\in C(\mathbb{R}_+,\mathbb{R}_+)$ such that
\begin{equation*}
\begin{array}{l}
\alpha_i(u,v),\beta_i(u,v)\le h(S)\left(\sum_{i=1}^m (u_i+v_i)^l+1\right)\\\text{for all }u,v\in\mathbb{R}_+^m.
\end{array}
\end{equation*}
\item[\textbf{(A6)}] There exists $K>0$ and $1\le r<3$ such that
\begin{equation*}
\begin{array}{l}
\text{for each }i=1,...,m,\text{ if }S\in\mathbb{R}_+\text{ and }u,v\in\mathbb{R}_+^m,\text{ then either}\\
-\dfrac{1}{y_{u,i}}\alpha_i(u,v)u_i+\beta_i(u,v)v_i\le K\left(\sum_{i=1}^m (u_i+v_i)^r+1\right)\\
\text{or}\\
\alpha_i(u,v)u_i-\dfrac{1}{y_{v,i}}\beta_i(u,v)v_i\le K\left(\sum_{i=1}^m (u_i+v_i)^r+1\right).
\end{array}
\end{equation*}
\end{enumerate}

\medskip\noindent Note that the power $l\ge1$ in (A5) above can be arbitrarily large, and the condition on $r$ in (A6) ultimately restricts the growth rate of at most one of $\alpha_i(u,v)$ or $\beta_i(u,v)$ for each $i=1,...,m$.

\medskip Note that the restriction on $r$ in (A6) leads to a sub-cubic intermediate sum in the sense discussed in \cite{FMTY} in the case of one space dimension, and as a result, this assumption should not be surprising. The recent work in \cite{KST} and \cite{STY} for reaction-diffusion models on bounded intervals might lead us to believe that $r=3$ is possible in (A6). However, careful analysis shows that the inclusion of the first spatial derivative terms in our system makes this impossible. The work in  \cite{KST} and \cite{STY} relies heavily on the spatial differential operators being constant multiples of one another, and the inclusion of the first derivative terms in (\ref{S1}) violates this. As a result, the arguments in \cite{KST} and \cite{STY} do not carry over to this setting, and we seem to be stuck with the sub cubic restriction listed above.

Our first result is stated below, and proved in the next section.

\medskip
\begin{theorem}\label{thm3}
Suppose (A1)-(A3), (A5) and (A6) are satisfied, and $y_{u,i}y_{v,i}\le1$ for all $i=1,...,m$. Then (\ref{S1})-(\ref{S3}) has a unique classical componentwise nonnegative global solution. Furthermore,
\begin{equation}\label{Sbound}
\|S(\cdot,t)\|_{\infty,(0,1)}=\|w_1(\cdot,t)\|_{\infty,(0,1)}\le \max\left\{\gamma_S,\|S_0\|_{\infty,(0,1)}\right\},
\end{equation} and if there exists $C>0$ independent of $\phi$ such that whenever
$$\|u_i(\cdot,t)\|_{1,(0,1)},\|v_i(\cdot,t)\|_{1,(0,1)}\le C\text{ for all }t>0\text{ and }i=1,...,m,$$
then there exists $\tilde{C}>0$ such that
$$\|S(\cdot,t)\|_{\infty,(0,1)},\|u_i(\cdot,t)\|_{\infty,(0,1)},\|v_i(\cdot,t)\|_{\infty,(0,1)}\le C\text{ for all }t>0\text{ and }i=1,...,m.$$
\end{theorem}

\medskip There are a few assumptions that guarantee uniform $L^1$ estimates for $u$ and $v$. One of these can be obtained by imposing the following assumption.
\medskip
\begin{enumerate}
\item[\textbf{(A7)}] $y_{u,i}y_{v,i}<1$ for all $i=1,...,m$, and there exists $\delta>0$ such that
\begin{equation*}
\begin{array}{l}
\delta\sum_{i=1}^m\left(u_i+v_i\right)\le \sum_{i=1}^m\left[\alpha_i(u,v)u_i+\beta_i(u,v)v_i\right]+1\\
\text{for all }u,v\in\mathbb{R}_+^m.
\end{array}
\end{equation*}
\end{enumerate}

Another assumption requires $y_{u,i}y_{v,i}$ to be sufficiently small for each $i=1,...,m$, based upon eigenfunctions associated with linear problems associated with our system. This was proved in \cite{ZA} for the system (\ref{Sa1})-(\ref{Sa3}), and it can be extended to the setting of (\ref{S1})-(\ref{S3}). To this end, for $d>0$, let $\lambda_{0,d}>1$ be the principle eigenvalue associated with
\begin{equation}\label{eigen}
\begin{array}{cc}
-d\phi''(x)-\phi'(x)=\lambda_{0,d}\phi(x),&0<x<1,\\
\phi'(0)=0,\,\,\,d\phi'(1)+\phi(1)=0.
\end{array}
\end{equation}
Then, for each $i\in\{1,...,m\}$ select principal eigen functions $\phi_{d_0}$, $\phi_{d_{i,u}}$ and $\phi_{d_{i,v}}$ associated with $d=d_0$, $d=d_{u,i}$ and $d=d_{v,i}$, respectively, so that
\begin{equation}\label{comparephi}
0<\phi_{d_{v,i}}(x)\le \phi_{d_{i,u}}(x)\le\phi_{d_0}(x),
\end{equation}
$\phi_{d_0}(0)=1$, and $\phi_{d_{v,i}}$ and $\phi_{d_{i,v}}$ are otherwise scaled as large as possible. Note that $$0<\min_{0\le x\le1}\frac{\phi_{d_{v,i}}(x)}{\phi_{d_{u,i}}(x)}<1\text{ when }d_{u,i}\ne d_{v,i}.$$

The following corollary is proved in the next section.

\begin{corollary}\label{Linfestimate}
Suppose (A1)-(A3), (A5) and (A6) are satisfied.
If either (A7) is true or $y_{u,i}y_{v,i}\le\frac{\phi_{d_{v,i}}(x)}{\phi_{d_{u,i}}(x)}$ for all $x\in[0,1]$, then there exists $C>0$ so that
$$\|S(\cdot,t)\|_{\infty,(0,1)}\|u_i(\cdot,t)\|_{\infty,(0,1)},\|v_i(\cdot,t)\|_{\infty,(0,1)}\le C$$
for all $i=1,...,m$ and $t\ge0$.
\end{corollary}

\medskip

One immediate question is whether finite time blow up can occur in (\ref{S1})-(\ref{S3}) when $y_{u,i}y_{v,i}>1$ for some $i\in\{1,...,m\}$. It turns out that this can occur, even if (A1)-(A3), (A5) and (A6) are satisfied.

\begin{theorem}\label{blowup}
Suppose $i=1$, $d_{u,1}=d_{v,1}$, (A1)-(A3), (A5) and (A6) are satisfied, and $y_{u,1}y_{v,1}>1$. If $\alpha(u,v)=\beta(u,v)=u_1+v_1$, then there exists $M>0$ so that if ${u_0}_1(x),{v_0}_2(x)\ge M$ for all $x\in[0,1]$, then the solution to (\ref{S1})-(\ref{S3}) blows-up in the sup norm in finite time.
\end{theorem}

\begin{remark}
It is possible to modify the proof in the next section to show that regardless of whether the diffusion coefficients are the same, if there exists $C>0$ so that
$$C\left(\alpha(u,v)u+\beta(u,v)v+1\right)=\left(u_1+v_1\right)^2,$$
and the initial data is sufficiently large, then the solution to (\ref{S1})-(\ref{S3}) blows-up in the sup norm in finite time.
\end{remark}

\section{Proofs of Theorem \ref{thm3}, Corollary \ref{Linfestimate} and Theorem \ref{blowup}}

\medskip The strategy for proving Theorem \ref{thm3} is to modify a greatly simplified version of the proofs in \cite{FMTY} to accomodate the first spatial derivative terms in (\ref{S1}), and the boundary conditions in (\ref{S2}). We begin by proving a slight modification of Lemma 2.3 from \cite{MT21} in the one dimensional setting.

\begin{lemma}\label{lm1}
Suppose $p\ge 2$ and $z^{p/2}\in H^1(0,1)$. If $\|z\|_{p/2,(0,1)}\le M$, $0\le\eta<2$ and $\varepsilon>0$, there exists $C_{p,\varepsilon,M}>0$, independent of $z$, so that
\begin{equation}\label{eta-ineq}
\int_0^1\left|z(x)\right|^{p+\eta}dx\le \varepsilon \int_0^1\left| \left(z(x)^{p/2}\right)_x\right|^2dx+C_{p,\varepsilon,M}.
\end{equation}
\end{lemma}
\begin{proof}
Let $w=z^{p/2}$. From the Gagliardo-Niremberg inequality, if $2\le q<4$ and $q=\frac{2}{2-3\theta}$, then $\frac{1}{3}\le\theta<\frac{1}{2}$, and there exists $C_q>0$ so that
$$
\int_0^1\left|w(x)\right|^qdx\le C_q\left(\left(\int_0^1\left|w_x(x)\right|^2dx\right)^{q\frac{\theta}{2}}\left(\int_0^1\left|w(x)\right|dx\right)^{q(1-\theta)}+\left(\int_0^1\left|w(x)\right|dx\right)^q\right).
$$
Also, $0<q\frac{\theta}{2}<1$. So, from above, Young's inequality implies there exists $C_{q,\varepsilon,M}>0$ so that
$$
\int_0^1\left|w(x)\right|^qdx\le \varepsilon\int_0^1\left| w_x(x)\right|^2dx+C_{q,\varepsilon,M}\left(1+\int_0^1\left|w(x)\right|dx\right)^\frac{2q(1-\theta)}{2-q\theta}.
$$
The result follows by choosing $q=2+\frac{2\eta}{p}$.
\end{proof}

\subsection{Proof of Theorem \ref{thm3}}
We assume w is the componentwise nonnegative global classical solution to (\ref{generalrdtr}), with everything defined as in the proof of Theorem \ref{thm2}. Also, we will use the names $S$, $u_i$ and $v_i$ instead of the $w_k$ names. Note that the maximum principle implies $S$ satisfies (\ref{Sbound}). We will obtain bounds for the $u_i$ and $v_i$ independent of the choice of the function $\phi$.

We start by deriving $L^1(0,1)$ bounds for $u_i$ and $v_i$ for each $i=1,...,m$, independent of $\phi$. We start by taking advantage of (\ref{figiL1}) and
(\ref{alphabetaL1}) to find
\begin{align}\label{figiL1more}
\int_0^T\int_0^1\left(\left|\phi f_i(S)u_i\right|+\left|\phi g_i(S)v_i\right|\right)dxdt\le \gamma_ST+\|S_0\|_{\infty,(0,1)}
\end{align}
and
\begin{align}\label{alphabetaL1more}
\int_0^T\int_0^1\left(\phi \alpha_i(u,v)u_{i}+\phi \beta_i(u,v)v_{i}\right)dxdt
\le &\delta\left(\left(\gamma_{u_i}+\gamma_{v_i}+\gamma_S\right)T\right.\nonumber\\&\left.+\|{u_0}_i+{v_0}_i\|_{\infty,(0,1)}+\|S_0\|_{\infty,(0,1)}\right)
\end{align}
for each $i=1,...,m$, where $$\delta>2\max_{j=1,...,m}\left\{\frac{y_{v,j}(1+y_{u,j})}{1-y_{u,j}y_{v,j}},\frac{y_{u,j}(1+y_{v,j})}{1-y_{u,j}y_{v,j}}\right\}.$$
Now, set
$$y_{\text{max}}=\max_{i=1,...,m}\left\{y_{u,i}.y_{v,i},1\right\},$$
and define
\begin{equation*}
\vec{v}=\left(y_{\text{max}},1+y_{u,1},1+y_{v,1},...,1+y_{u,m},1+y_{v,m}\right).
\end{equation*}
Then integrating $\vec{v}\cdot w$ (i.e. $\vec{v}\cdot (S,u_1,v_1,...,u_m,v_m)$) over $(0,1)\times(0,T)$ results in a bound for $\|u_i(\cdot,t)\|_{1,(0,1)}$ and $\|v_i(\cdot,t)\|_{1,(0,1)}$, for all $i=1,...,m$ in terms of a continuous function of $t\ge0$ that is independent of $\phi$.

Now we derive $L^p(0,1)$ estimates for $u_i$ and $v_i$ for each $2\le p<\infty$ and $i=1,...,m$. Without loss of generality, we assume (A6) holds with the first option for each $i=1,...,m$. That is, we assume there exists $K>0$ and $1\le r<3$ such that
\begin{equation}\label{wlog}
\left\{\begin{array}{l}
\text{for each }i=1,...,m,\text{ if }S\in\mathbb{R}_+\text{ and }u,v\in\mathbb{R}_+^m,\text{ then }\\
-\dfrac{1}{y_{u,i}}\alpha_i(u,v)u_i+\beta_i(u,v)v_i\le K\left(\sum_{i=1}^m (u_i+v_i)^r+1\right).\\
\end{array}\right.
\end{equation}
Choose $a\ge1$ such that
\begin{equation}\label{a1stuff}
a\ge \max\left\{y_{u,i},\frac{d_{u,i}+d_{v,i}}{2\sqrt{d_{u,i}d_{v,i}}}\right\}
\end{equation}
for each $i=1,...,m$. Note that (\ref{a1stuff}) implies
the matrix $$\left(\begin{array}{cc}d_{u,i}a^2&\frac{d_{u,i}+d_{v,i}}{2}\\\frac{d_{u,i}+d_{v,i}}{2}&d_{v,i}\end{array}\right)$$ is positive definite  for each $i=1,...,m$, and, combining this with (\ref{wlog}), for each $q\ge1$ there exists $K_q>0$ so that
\begin{align}\label{a2stuff}
a^qf_i(S)u_i+g_i(S)v_i+a^q&\left(-\frac{1}{y_{u,i}}\alpha_i(u,v)u_i+\beta_i(u,v)v_i\right)+\left(\alpha_i(u,v)u_i-\frac{1}{y_{v,i}}\beta_i(u,v)v_i\right)\nonumber\\
&\le K_q\left(\sum_{j=1}^m (u_j+v_j)^r+1\right).
\end{align}

Now, construct an $L^p$-energy functional. We write $\mathbb{Z}_{+}^{2}$ for the set of all $2$-tuples of non negative integers. Addition and
scalar multiplication by non negative integers of elements in $\mathbb{Z}_{+}^{2}$
is understood in the usual manner. If $\beta=(\beta_{1},\beta_{2})\in \mathbb{Z}_{+}^{2}$ and $p\in \mathbb N$,
then we define $\beta^{p}=((\beta_{1})^{p},(\beta_{2})^{p})$.
Also, if $\alpha=(\alpha_{1},\alpha_{2})\in  \mathbb{Z}_{+}^{2}$, then
we define $|\alpha|=\alpha_{1}+\alpha_2$. Finally, if $p$ is a positive integer, we define
$$\begin{pmatrix}
	p\\ \beta\end{pmatrix}=\frac{p!}{\beta_1!\beta_2!}.$$

 Let $i\in\{1,...,\}$ and $2\leq p\in \mathbb N$. We build our $L^p$-energy functional
\begin{equation}\label{Lp}
	\L_{p,i}[u_i,v_i](t) = \int_0^1 \H_{p,i}[u_i,v_i](t),
\end{equation}
where
\begin{equation}\label{Hp}
	\H_{p,i}[u_i,v_i](t) = \sum_{\beta\in \mathbb Z_+^{2}, |\beta| = p}\begin{pmatrix}
	p\\ \beta\end{pmatrix}a^{\beta_1^2}u_i^{\beta_1}v_i^{\beta_2},
\end{equation}
with the value $a$ chosen as above. For convenience, we drop the subscript $\beta\in \mathbb Z_+^{2}$ in the sums below, as it should be clear.

The following result can be found in \cite{MT21}.

\begin{lemma}\label{Hp-lem7}
	Suppose $\H_{p,i}[u_i,v_i]$ is defined in (\ref{Hp}). Then
	\begin{equation*}
	\frac{\partial}{\partial t}\H_{p,i}[u_i,v_i](t) = \sum_{|\beta| = p-1}\begin{pmatrix} p\\ \beta \end{pmatrix} a^{\beta_1^2}u_i^{\beta_1}v_i^{\beta_2}\left(a^{2\beta_{1}+1}\frac{\partial}{\partial t}u_{i}+\frac{\partial}{\partial t}v_{i}\right),
	\end{equation*}
and
\begin{align*}
\sum_{|\beta|=p-1}\begin{pmatrix}p\\ \beta\end{pmatrix}a^{\beta_1^2}&\left(a^{2\beta_{1}+1}d_{u,i}{u_i}_{x}+d_{v,i}{v_i}_{x}\right)\left( u_i^{\beta_1}v_i^{\beta_2}\right)_x \\&=\sum_{|\beta|=p-2}\begin{pmatrix}p\\ \beta\end{pmatrix}a^{\beta_1^2}u_i^{\beta_1}v_i^{\beta_2}\left({u_i}_x\quad{v_i}_x\right)A_i\left(\begin{array}{c}{u_i}_x\\{v_i}_x\end{array}\right),
\end{align*}
	where
\begin{equation*}
A_i=\left(
\begin{array}{cc}
d_{u,i}a^{4\beta_1+4}&\frac{d_{u,i}+d_{v,i}}{2}a^{2\beta_1+1}\\
\frac{d_{u,i}+d_{v,i}}{2}a^{2\beta_1+1}&d_{v,i}
\end{array}
\right).
\end{equation*}
\end{lemma}

\bigskip

\bigskip Now choose $p\ge2$ such that $u_j(\cdot,t),v_j(\cdot,t)\in L_{p/2}(0,1)$, with norms bounded above by a continuous function of $t$ for $t\ge0$. From above, we are guaranteed that $p=2$ is such a value. From Lemma \ref{Hp-lem7},

\begin{equation}\label{f5}
\bra{\L_{p,i}[u_i,v_i]}'(t)=\int_0^1\sum_{|\beta|=p-1}\left(\begin{array}{c}
p\\
\beta
\end{array}\right)a^{\beta_1^{2}}u_i^{\beta_1}v_i^{\beta_2}\left(a^{2\beta_{1}+1}\frac{\partial}{\partial t}u_{i}+\frac{\partial}{\partial t}v_{i}\right)dx
\end{equation}
Note that from (\ref{wlog}), (\ref{a1stuff}) and (\ref{a2stuff}), we get
\begin{align}\label{uivi}
a^{2\beta_{1}+1}\frac{\partial}{\partial t}u_{i}+\frac{\partial}{\partial t}v_{i}\le a^{2\beta_{1}+1}\left(d_{u,i}{u_i}_{xx}-{u_i}_x\right)&+\left(d_{v,i}{v_i}_{xx}-{v_i}_x\right)\\
&+K_{2\beta_1+1}\left(\sum_{j=1}^m(u_j+v_j)^r+1\right).
\end{align}
The function $\phi$ does not appear above because $0\le\phi(w)\le1$. This allows us to rewrite (\ref{f5}) in the form
\begin{equation}\label{f6}
\bra{\L_{p,i}[u_i,v_i]}'(t)=(I)+(II),
\end{equation}
where
$$(I)=\int_0^1\sum_{|\beta|=p-1}\left(\begin{array}{c}
p\\
\beta
\end{array}\right)a^{\beta_1^{2}}u_i^{\beta_1}v_i^{\beta_2}\left(a^{2\beta_{1}+1}\left(d_{u,i}{u_i}_{xx}-{u_i}_x\right)+\left(d_{v,i}{v_i}_{xx}-{v_i}_x\right)\right)dx$$
and
$$
(II)=\int_0^1\sum_{|\beta|=p-1}\left(\begin{array}{c}
p\\
\beta
\end{array}\right)a^{\beta_1^{2}}u_i^{\beta_1}v_i^{\beta_2}K_{2\beta_1+1}\left(\sum_{j=1}^m(u_j+v_j)^r+1\right)dx.
$$
Our boundary conditions, integration by parts and Lemma \ref{Hp-lem7} imply
\begin{align}\label{Iinfo}
(I)=-&\int_0^1\sum_{|\beta|=p-2}\begin{pmatrix}p\\ \beta\end{pmatrix}a^{\beta_1^2}u_i^{\beta_1}v_i^{\beta_2}\left({u_i}_x\quad{v_i}_x\right)A_i\left(\begin{array}{c}{u_i}_x\\{v_i}_x\end{array}\right)dx+(III)\nonumber\\
&+\int_0^1\sum_{|\beta|=p-1}\left(\begin{array}{c}
p\\
\beta
\end{array}\right)a^{\beta_1^{2}}u_i^{\beta_1}v_i^{\beta_2}\left(a^{2\beta_{1}+1}{u_i}+{v_i}\right)dx,
\end{align}
where
\begin{equation}\label{bcinfo}
(III)=\sum_{|\beta|=p-1}\left(\begin{array}{c}
p\\
\beta
\end{array}\right)a^{\beta_1^{2}}u_i(0,t)^{\beta_1}v_i(0,t)^{\beta_2}\left(a^{2\beta_{1}+1}\gamma_{u,i}+\gamma_{v,i}\right).
\end{equation}
From the choice of $a$ and Lemma \ref{Hp-lem7}, there exists $c_p,\delta_p>0$ independent of $i$ so that
\begin{align}\label{I2info}
(I)\le-2\delta_p&\int_0^1\left(\left|\left({u_i}^{p/2}\right)_x\right|^2+ \left|\left({v_i}^{p/2}\right)_x\right|^2\right)dx\nonumber\\
&+c_p\left(u_i(0,t)^{p-1}+v_i(0,t)^{p-1}
+\int_0^1\left({u_i}^{p}+ {v_i}^{p}\right)dx\right)
\end{align}
and
\begin{equation}\label{IIinfo}
(II)\le c_p\int_0^1 \sum_{j=1}^m\left(u_j^{p+r-1}+v_j^{p+r-1}+1\right)dx.
\end{equation}
Now, note that if $\varepsilon>0$, there exits $C_\varepsilon>0$ so that
\begin{align*}
u_i(0,t)^{p-1}+v_i(0,t)^{p-1}&\le C_\varepsilon+\varepsilon \left(u_i(0,t)^{p}+v_i(0,t)^{p}\right)\nonumber\\
&\le C_\varepsilon+\varepsilon K\|u_i^{p/2}+v_i^{p/2}\|_{H^1(0,1)}^2.
\end{align*}
for some $K>0$. As a result, since $r\ge1$, we can choose $\varepsilon$ sufficiently small and $C_p>0$ sufficiently large that
\begin{align}\label{I-IIinfo}
\bra{\L_{p,i}[u_i,v_i]}'(t)\le -\delta_p&\int_0^1\left(\left|\left({u_i}^{p/2}\right)_x\right|^2+ \left|\left({v_i}^{p/2}\right)_x\right|^2\right)dx\nonumber\\
&+C_p\int_0^1 \sum_{j=1}^m\left(u_j^{p+r-1}+v_j^{p+r-1}+1\right)dx.
\end{align}
Now, define
$$
\L_{p}[u,v](t)=\sum_{i=1}^m\L_{p,i}[u_i,v_i](t).
$$
Then (\ref{I-IIinfo}) implies
\begin{align}\label{fulldeal}
\bra{\L_{p}[u,v]}'(t)\le -\delta_p&\int_0^1\sum_{i=1}^m\left(\left|\left({u_i}^{p/2}\right)_x\right|^2+ \left|\left({v_i}^{p/2}\right)_x\right|^2\right)dx\nonumber\\
&+mC_p\int_0^1 \sum_{i=1}^m\left(u_i^{p+r-1}+v_i^{p+r-1}+1\right)dx.
\end{align}
As a result, since there exits $M\in C(\mathbb{R}_+,\mathbb{R}_+)$ such that
$$
\sum_{i=1}^m\left(\|u_i(\cdot,t)\|_{p/2,(0,1)}+\|v_i(\cdot,t)\|_{p/2,(0,1)}\right)\le M(t)\text{ for all }t\ge0,
$$
we can conclude from Lemma \ref{lm1}, (\ref{fulldeal}) and the definition of $\L_{p,i}[u,v](t)$, for each $i=1,...,m$, that there exist $M_p\in C(\mathbb{R}_+,\mathbb{R}_+)$ and $e_p>0$ such that
\begin{align}\label{fulldeal2}
\bra{\L_{p}[u,v]}'(t)\le &M_p(t) -e_p\L_{p}[u,v](t).
\end{align}
Furthermore, $\|M_p\|_{\infty,(0,\infty)}<\infty$ if $\|M\|_{\infty,(0,\infty)}<\infty$. It follows that there exists $N_p\in C(\mathbb{R}_+,\mathbb{R}_+)$ such that
\begin{align}\label{fulldeal3}
\sum_{i=1}^m\left(\|u_i(\cdot,t)\|_{p,(0,1)}+\|v_i(\cdot,t)\|_{p,(0,1)}\right)\le &N_p(t)\text{ for all }t\ge0,
\end{align}
where $\|N_p\|_{\infty,(0,\infty)}<\infty$ if $\|M\|_{\infty,(0,\infty)}<\infty$, and this bound is independent of $\phi$. From induction, this result holds for all $2\le p<\infty$. As a result, from (A6) and the results in \cite{LSU68,Nit}, we can conclude there exists $N_\infty\in C(\mathbb{R}_+,\mathbb{R}_+)$ such that
\begin{align}\label{fulldeal4}
\sum_{i=1}^m\left(\|u_i(\cdot,t)\|_{\infty,(0,1)}+\|v_i(\cdot,t)\|_{\infty,(0,1)}\right)\le &N_\infty(t)\text{ for all }t\ge0,
\end{align}
where $\|N_\infty\|_{\infty,(0,\infty)}<\infty$ if $\|M\|_{\infty,(0,\infty)}<\infty$, and this bound is independent of $\phi$. As a result, from Corollary \ref{Cor1}, the system (\ref{S1})-(\ref{S3}) has a unique classical componentwise nonnegative global solution, and the bounds above imply the conclusion of Theorem \ref{thm3}.

\subsection{Proof of Corollary \ref{Linfestimate}}
From Theorem \ref{thm3}, (\ref{S1})-(\ref{S3}) has a unique classical componentwise nonnegative global solution, and the result of Corollary \ref{Linfestimate} will be confirmed if we can prove there exists $C>0$ such that
$$\|u_i(\cdot,t)\|_{1,(0,1)},\|v_i(\cdot,t)\|_{1,(0,1)}\le C\text{ for all }t\ge0\text{ and }i=1,...,m.$$

First, suppose (A7) is true. Define
$$y_{\max}=m\max_{i=1,...,m}\left\{y_{u,i}+1,y_{v,i}+1\right\}$$
and
$$a=\left(y_{\max},y_{u,1}+1,y_{v,1}+1,y_{u,2}+1,y_{v,2}+1,...,y_{u,m}+1,y_{v,m}+1\right).$$
If
$$0<\varepsilon=\min_{i=1,...,m}\left\{\frac{1-y_{u,i}y_{v,i}}{y_{u,i}},\frac{1-y_{u,i}y_{v,i}}{y_{v,i}}\right\},\text{ }
a_{\text{max}}=\max_{i=1,...,m}a_i\text{ and } L=\max_{i=1,...,m}\left\{\gamma_S,\gamma_{u,i},\gamma_{v,i}\right\},$$
then from (A7),
\begin{align}\label{A7L1}
\frac{d}{dt}\int_0^1 a\cdot\left(S(x,t),u_1(x,t),v_1(x,t),...,u_m(x,t),v_m(x,t)\right)dx\le& \varepsilon+a_{\text{max}}L\nonumber\\
-\varepsilon\delta\int_0^1\sum_{i=1}^m\left(u_i(x,t)+v_i(x,t)\right)dx.
\end{align}
Since we have uniform sup norm bound for $S$, (\ref{A7L1}) implies there exists $C>0$ so that
$$\|u_i(\cdot,t)\|_{1,(0,1)},\|v_i(\cdot,t)\|_{1,(0,1)}\le C\text{ for all }t\ge0\text{ and }i=1,...,m.$$
Consequently, Theorem \ref{thm3} guarantees a uniform sup norm bound for $u_i$ and $v_i$ for each $i=1,...,m$.

Now, suppose $y_{u,i}y_{v,i}\le\frac{\phi_{d_{v,i}}(x)}{\phi_{d_{u,i}}(x)}$ for all $x\in[0,1]$. We apply a slight modification of the proof of Theorem 4 in \cite{ZA} to obtain a uniform $L^1(0,1)$ estimate for each $u_i$ and $v_i$ as follows. First let $i\in\{1,...,\}$ and as in \cite{ZA}, define
$$
X(t)=\int_0^1S(x,t)\phi_{d_0}(x)dx,\,Y(t)=\int_0^1u_i(x,t)\phi_{d_{u,i}}(x)dx,\,Z(t)=\int_0^1v_i(x,t)\phi_{d_{v,i}}(x)dx.
$$
Then, integration by parts implies there exists $L>0$ so that
$$
X'(t)\le C-\lambda_{d_0}X(t)-\int_0^1\left(f_i(S)u_i(x,t)+g_i(S)v_i(x,t)\right)dx,
$$
$$
Y'(t)\le C-\lambda_{d_{u,i}}Y(t)+\int_0^1\left(f_i(S)u_i(x,t)-\frac{1}{y_{u,i}}\alpha(u,v)u_i(x,t)+\beta_i(u,v)v_i(x,t)\right)\phi_{d_{u,i}}dx
$$
and
$$
Z'(t)\le C-\lambda_{d_{v,i}}Z(t)+\int_0^1\left(g_i(S)v_i(x,t)+\alpha(u,v)u_i(x,t)-\frac{1}{y_{v,i}}\beta_i(u,v)v_i(x,t)\right)\phi_{d_{v,i}}dx.
$$
Then it is possible to choose $\varepsilon>0$ sufficiently small and $\theta=\varepsilon y_{u,i}$ to guarantee
$$
\varepsilon\le \frac{\phi_{d_0}}{\phi_{d_{v,i}}},\,\theta\le \frac{\phi_{d_0}}{\phi_{d_{u,i}}},\, \frac{\varepsilon y_{u,i}}{\theta}\le \frac{\phi_{d_{u,i}}}{\phi_{d_{v,i}}}\le \frac{\varepsilon}{\theta y_{v,i}}
$$
and
$$Q(t)=X(t)+\theta Y(t)+\varepsilon Z(t),$$
then there exist $a,b>0$ so that
$$Q'(t)\le a-b Q(t).$$
The result follows.

\subsection{Proof of Theorem \ref{blowup}}
Since $i=1$, we can assume we are working with system (\ref{Sa1})-(\ref{Sa3}), and denote $d=d_1=d_2$. As a result, $d=d_{u,1}=d_{v,1}$, and we write $y_u$, $y_v$, $f(S)$, $g(S)$, $u$, $v$, $\alpha$ and $\beta$ for $y_{u,1}$, $y_{v,1}$, $f_1(S)$, $g_1(S)$, $u_1$, $v_1$, $\alpha_1$ and $\beta_1$. Throughout, we assume $S,u,v$ is our unique componentwise nonnegative classical maximal solution of (\ref{Sa1})-(\ref{Sa3}). Recall that our hypothesis implies
$$\alpha(u,v)=\beta(u,v)=u+v.$$
Let $\lambda_{0,d}>1$ be the principle eigenvalue associated with (\ref{eigen}), and choose an associated eigenfunction $\phi(x)$ with $0<\phi(x)\le1$ for all $x\in[0,1]$. Then, similar to the proof of Corollary \ref{Linfestimate}, we define
$$Y(t)= \int_0^1u(x,t)\phi(x)dx,\text{ and }Z(t)=\int_0^1v(x,t)\phi(x)dx.$$
Then, integration by parts implies
$$Y'(t)=-\lambda_{d,0}Y(t)+\int_0^1\left(f(S)u-\frac{1}{y_u}\alpha(u,v)u+\beta(u,v)v\right)\phi dx$$
and
$$Z'(t)=-\lambda_{d,0}Z(t)+\int_0^1\left(g(S)u+\alpha(u,v)u-\frac{1}{y_v}\beta(u,v)v\right)\phi dx.$$
Consequently, if we define
$$Q(t)=(y_u+1)Y(t)+(y_v+1)Z(t),$$
then
$$Q'(t)\ge-\lambda_{d,0}Q(t)+\frac{y_uy_v-1}{y_u}\int_0^1\alpha(u,v)u\phi dx+\frac{y_uy_v-1}{y_v}\int_0^1\beta(u,v)v\phi dx.$$
Now, let
$$\varepsilon=\min\left\{\frac{y_uy_v-1}{y_u},\frac{y_uy_v-1}{y_v}\right\}.$$
Since $\alpha(u,v)=\beta(u,v)=u+v$ and $0<\phi(x)\le1$, we have
$$Q'(t)\ge -\lambda_{d,0}Q(t)+\varepsilon\int_0^1\left(u+v\right)^2\phi dx.$$
Therefore, there exists $\tilde\varepsilon>0$ so that
$$Q'(t)\ge -\lambda_{d,0}Q(t)+\tilde\varepsilon Q(t)^2.$$
Consequently, if $$Q(0)>\frac{\lambda_{d,0}}{\tilde\varepsilon},$$ then $Q(t)$ blows up in finite time. The result follows.

\medskip

\section{Steady State Results for (\ref{Sa1})-(\ref{Sa2})}
In this section, we are interested in the existence of a non trivial steady-state solution with microorganisms present in the medium, in the case $\gamma_{S}=1,\gamma_{u}=0,\gamma_{v}=0$. So, we show that there exist $u$ or $v$ that behaves like the principal eigenfunction of problem $(P_{\lambda})$ below, associated respectively with $\lambda_{d_{1}},\lambda_{d_{2}}$, that is, $$u\approx \phi_{d_{1}}\,or\,v\approx \phi_{d_{2}},$$
in the sense that there exists a positive constant $c\geq1$ such that $\frac{1}{c} \phi_{d_{1}}(x)\leq u(x)\leq c\phi_{d_{1}}(x)$ or $\frac{1}{c} \phi_{d_{2}}(x)\leq v(x)\leq c\phi_{d_{2}}(x),\,\text{ for all, } 0\leq x\leq1.$

To obtain our result, we consider the eigenvalue problem introduced in \cite{BL}:
\begin{equation*}
(P_{\lambda }) \left \{
\begin{array}{l}
\displaystyle \lambda \phi =d\phi ''-\phi '.
\\
\displaystyle -d\phi '(0)+\phi (0)=0,\,\,\phi '(1)=0,
\end{array}%
\right .
\end{equation*}
where $d$ is a positive constant. The eigenvalues
$\{\lambda _{n}\}_{n\geq 0}$ of $(P_{\lambda })$ satisfy
$\lambda _{n+1}<\lambda _{n},\,\,\forall n\geq 0$ and
$\lambda _{0}<-1$. To emphasize the dependence of
$\lambda _{0}$ on $d$ and consider its sign, we use
$\lambda _{d}=-\lambda _{0}$, and we denote by $\phi _{d}$ the principal eigenfunction
associated with $\lambda _{d}$.%

Our results are obtained by means of Schauder fixed point theorem combined with spectral theory.
So, it is convenient to make the change of variables $\widetilde{S}=1-S$ and we will always interpret $1-\widetilde{S}$ as the positive part of it: $(1-\widetilde{S})_{+}$, so that $\widetilde{S}$ satisfies homogeneous boundary conditions. We then have, the following steady-state system,
\begin{equation}\label{SS1}
\left\{
\begin{array}{l}
-d_{0}\widetilde{S}_{xx}+\widetilde{S}_{x}=f(1-\widetilde{S})u+g(1-\widetilde{S})v\\
-d_{1}u_{xx}+u_{x}=f(1-\widetilde{S})u+\beta(u,v)v-\frac{1}{y_{u}}\alpha(u,v)u\\
-d_{2}v_{xx}+v_{x}=g(1-\widetilde{S})v+\alpha(u,v)u-\frac{1}{y_{v}}\beta(u,v)v\\
\end{array}%
\right.
\end{equation}%
with boundary conditions
\begin{equation}\label{SS1b}
\left\{
\begin{array}{l}
-d_{0}\widetilde{S}_{x}(0)+\widetilde{S}(0)=0,\\
-d_{1}u_{x}(0)+u(0)=-d_{2}v_{x}(0,t)+v(0)=0,\\
\widetilde{S}_{x}(1)=0,\,\,u_{x}(1)=0,\,\,v_{x}(1)=0.\\
\end{array}%
\right.
\end{equation}
We can convert the differential operator defined by the left side of system
(\ref{SS1})-(\ref{SS1b}), and consider (\ref{SS1}) to be the fixed point
equation
\begin{equation*}
(\widetilde{S},u,v)=T(\widetilde{S},u,v),
\end{equation*}
where $T$ is defined by the right side of (\ref{eqp11})
on the positive cone in $(C([0,1],\mathbb{R}^{+}))^{3}$:%
%
\begin{equation}
\label{eqp11}
\begin{array}{llll}
&\widetilde{S}(x)=\displaystyle \int _{0}^{1}\exp (
\frac{\min (x,t)-t}{d_{0}})[f(1-\widetilde{S})u+g(1-\widetilde{S})v]dt,
\\
&u(x)=\displaystyle \int _{0}^{1}\exp (
\frac{\min (x,t)-t}{d_{1}})[f(1-\widetilde{S})u+\beta (u,v)v-
\frac{1}{y_{u}}\alpha (u,v)u]dt,
\\
&v(x)=\displaystyle \int _{0}^{1}\exp (
\frac{\min (x,t)-t}{d_{2}})[g(1-\widetilde{S})v+
\alpha (u,v)u-\frac{1}{y_{v}}\beta (u,v)v]dt.
\end{array}
\end{equation}
In what follows we replace the above $\widetilde{S}$ by $S$. Our first main result is the following:
\begin{theorem}
  (Extinction of attached bacteria)\\
 Assume that $\alpha(\cdot, 0)=0$ and $f$ is a non-decreasing continuous function satisfying $e^{\frac{1}{d_{1}}}<f(1)\leq\lambda_{d_{1}} $. Then (\ref{SS1})-(\ref{SS1b}) has non trivial solution $(S,u,0)$ satisfying $S<1$ and $u>0.$\\
 (Extinction of isolated bacteria)\\
 Assume $\beta(0,\cdot)=0$ and $g$ is a non-decreasing continuous function satisfying $e^{\frac{1}{d_{2}}}<g(1)\leq\lambda_{d_{2}}$, then (\ref{SS1})-(\ref{SS1b}) has non trivial solution $(S,0,v)$ satisfying $S<1$ and $v>0.$
\end{theorem}
\begin{proof}(Extinction of attached bacteria)
Let $\phi _{d}>0$ be the principal eigenfunction of the Sturm--Liouville
problem $(P_{\lambda })$ corresponding to the eigenvalue $-\lambda _{d}$. We
denote by
$\phi _{d_{0}},\,\,\phi _{d_{1}}$ the eigenfunctions
associated respectively with
$-\lambda _{d_{0}},\,\,-\lambda _{d_{1}}$. Since, $f(1)> e^{\frac{1}{d_{1}}},$ then there exist $0<k<1$, satisfying that $f(1-k)> e^{\frac{1}{d_{1}}}$. We normalize $\phi _{d_{0}}$ by requiring that
$\phi _{d_{0}}<k$, and normalize $\phi _{d_{1}}$
by requiring that:
%
\begin{equation}
\label{norm1}
 \phi _{d_{1}}(x)\leq\frac{\lambda_{d_{0}}}{f(1)} \phi _{d_{0}}(x).
\,\,0\leq x\leq 1.
\end{equation}
  Let $\displaystyle c=\min_{[0,1]}\phi _{d_{1}}$ and let us define the cone $K$ by $$K:=\{(S,u)\in (C([0,1],\mathbb{R}))^{2};0\leq S\leq \phi _{d_{0}},\,c\leq u\leq \phi _{d_{1}}\}$$ and the operator $\mathfrak{T}=(T_{1},T_{2})$ on $K$ by \begin{equation}
\label{eqp110}
\begin{array}{llll}
&T_{1} (S,u)(x)=\displaystyle \int _{0}^{1}\exp (
\frac{\min (x,t)-t}{d_{0}})f(1-S)u dt,
\\
&T_{2} (S,u)(x)=\displaystyle \int _{0}^{1}\exp (
\frac{\min (x,t)-t}{d_{1}})f(1-S)u dt.

\end{array}
\end{equation}
Indeed, thanks to (\ref{norm1}), we obtain, for all $(S,u)\in K,$ that\\
$$0\leq T_{1} (S,u)(x)\leq f(1)\displaystyle \int _{0}^{1}\exp (
\frac{\min (x,t)-t}{d_{0}})\phi_{d_{1}}dt\leq\phi_{d_{0}}(x),$$
which implies $0\leq T_{1} (S,u)\leq \phi_{d_{0}}.$
Furthermore, since we have $\|\phi_{d_{0}}\|<k<1$, we obtain $$T_{2} (S,u)(x)\geq \displaystyle\int _{0}^{1}\exp (
\frac{\min (x,t)-t}{d_{1}})f(1-k)cdt,$$ so we have $$ce^{\frac{-1}{d_{1}}}f(1-k)\leq T_{2} (S,u)(x)\leq f(1)\displaystyle \int _{0}^{1}\exp (
\frac{\min (x,t)-t}{d_{1}})\phi_{d_{1}}dt=f(1)\frac{\phi_{d_{1}}(x)}{\lambda_{d_{1}}}.$$
Consequently, $$c\leq T_{2} (S,u)(x)\leq \phi_{d_{1}}(x).$$
 Moreover, it follows from well-known arguments in \cite{A} that
$T$ is a compact continuous operator on $K$, and we have that $\mathfrak{T}(K)\subset K,$ so we deduce using Schauder fixed point theorem that there exists $(S,u)\in K$ satisfying that $\mathfrak{T}(S,u)=(S,u)$; that is $T_{1} (S,u)=S$ and $T_{2} (S,u)=u$. And since $\alpha(\cdot, 0)=0$, this implies $(S,u,0)$ is a non trivial solution of system (\ref{SS1})-(\ref{SS1b}).\\

(Extinction of isolated bacteria)
Let $\phi _{d_{2}}>0$ be the principal eigenfunction of the Sturm--Liouville
problem $(P_{\lambda })$ corresponding to the eigenvalue $-\lambda _{d_{2}}$. Since, $g(1)> e^{\frac{1}{d_{1}}},$ there exist $0<k<1$, satisfying $g(1-k)> e^{\frac{1}{d_{2}}}$. We normalize $\phi _{d_{0}}$ by requiring that
$\phi _{d_{0}}<k$, and normalize $\phi _{d_{2}}$
by requiring
%
\begin{equation}
\label{norm2}
 \phi _{d_{2}}(x)\leq\frac{\lambda_{d_{0}}}{g(1)} \phi _{d_{0}}(x).
\,\,0\leq x\leq 1.
\end{equation}
 Let $C=\min_{[0,1]}\phi _{d_{2}}$ and define the cone $\Lambda$ by $$\Lambda:=\{(S,u)\in (C([0,1],\mathbb{R}))^{2};0\leq S\leq \phi _{d_{0}},\,C\leq v\leq \phi _{d_{2}}\}$$ and the operator $F=(F_{1},F_{2})$ on $\Lambda$ by \begin{equation}
\label{eqp110}
\begin{array}{llll}
&F_{1} (S,v)(x)=\displaystyle \int _{0}^{1}\exp (
\frac{\min (x,t)-t}{d_{0}})g(1-S)vdt,
\\
&F_{2} (S,v)(x)=\displaystyle \int _{0}^{1}\exp (
\frac{\min (x,t)-t}{d_{2}})g(1-S)vdt.

\end{array}
\end{equation}
Indeed, thanks to (\ref{norm2}), we obtain, for all $(S,v)\in \Lambda,$ that\\
$$0\leq F_{1} (S,v)(x)\leq g(1)\displaystyle \int _{0}^{1}\exp (
\frac{\min (x,t)-t}{d_{0}})\phi_{d_{2}}dt\leq\phi_{d_{0}}(x),$$
and this implies $0\leq F_{1} (S,v)\leq \phi_{d_{0}}.$
Furthermore, we have $$e^{\frac{-1}{d_{2}}}g(1-k)C\leq F_{2} (S,v)(x)\leq g(1)\displaystyle \int _{0}^{1}\exp (
\frac{\min (x,t)-t}{d_{2}})\phi_{d_{2}}dt=g(1)\frac{\phi_{d_{2}}(x)}{\lambda_{d_{2}}}.$$
As a result, $$C\leq F_{2} (S,v)(x)\leq \phi_{d_{2}}(x).$$
 Moreover, it follows from well-known arguments that
$F$ is a compact continuous operator on $\Lambda$, and we have that $F(\Lambda)\subset \Lambda$. So we deduce using Schauder fixed point theorem that there exists $(S,v)\in \Lambda$ satisfying that $F(S,v)=(S,v)$. That is, $F_{1} (S,u)=S$ and $F_{2} (S,v)=v$. And since $\beta(0,\cdot)=0$, this implies $(S,0,v)$ is a non trivial solution of system (\ref{SS1})-(\ref{SS1b}).\\

\end{proof}
\begin{remark}
  If we integrate the first equation of (\ref{SS1}) on $[0,1]$ with $v=0,$ we obtain $$S(1)=\displaystyle\int_{0}^{1}f(1-S)u.$$
Since $\displaystyle T_{2} (S,u)(x)\geq e^{\frac{-1}{d_{1}}}\int_{0}^{1}f(1-S)u=e^{\frac{-1}{d_{1}}}S(1)$, this
 yields $S(1)=0,$ and then $u\equiv 0$ or $S(1)\neq 0,$ and consequently $u>0.$\\
\end{remark}

Let us define the cone $\Sigma$ by $$\Sigma:=\{(S,u,v)\in (C([0,1],\mathbb{R}))^{3};0\leq S\leq \phi _{d_{0}},\,\frac{\phi _{d_{1}}}{\lambda_{d_{1}}}\leq u\leq \phi _{d_{1}},\,\frac{\phi _{d_{2}}}{\lambda_{d_{2}}}\leq v\leq \phi _{d_{2}}\}.$$
Here $\phi _{d_{0}},\,\,\phi _{d_{1}}$ and $\phi _{d_{2}}$ are eigenfunctions
associated respectively with
$-\lambda _{d_{0}},\,\,-\lambda _{d_{1}}$ and $-\lambda _{d_{2}}$ satisfying some appropriate conditions.
 Choose $\theta>0$ small enough satisfying $\theta\phi_{d_{1}}\leq\phi_{d_{2}}$ and for all $(S,u,v)\in \Sigma,$ $\theta\leq\alpha(u,v)$ and choose $\rho>0$ small enough satisfying for all $(S,u,v)\in \Sigma,$ $\rho\leq\beta(u,v)$.
 Our second main result is the following:
\begin{theorem}
  (Coexistence of isolated and attached bacteria)\\
 Suppose $\alpha $ and $\beta $ are positive
continuous functions and nondecreasing in $(u,v)$, and assume $f,g$ are nondecreasing continuous functions satisfying $$ f(1)>\lambda_{d_{1}}(1+\frac{\alpha (1,1)}{y_{u}}),\,g(1)>\lambda_{d_{2}}(1+\frac{\beta (1,1)}{y_{v}})$$ and let $y_{u}$ and $y_{v}$ satisfy the assumptions
   \begin{equation}\label{A2}
           f(1)+\beta(1,1)\leq \lambda_{d_{1}}+\frac{\theta}{y_{u}\lambda_{d_{1}}}
        \end{equation}
        and
   \begin{equation}\label{A4}
           g(1)+\frac{\alpha(1,1)}{\theta}\leq \lambda_{d_{2}}+\frac{\rho}{y_{v}\lambda_{d_{2}}}.
        \end{equation}
 Then (\ref{SS1})-(\ref{SS1b}) has non trivial solution $(S,u,v)$ satisfying $S<1,\,u>0,\,\,v>0.$
\end{theorem}
\begin{proof}
First, since $f(1)>\lambda_{d_{1}}(1+\frac{\alpha (1,1)}{y_{u}})$ and $g(1)>\lambda_{d_{2}}(1+\frac{\beta (1,1)}{y_{v}})$, there exists $0<k,k^{\prime}<1$, such that
$$\lambda_{d_{1}}(1+\frac{\alpha (1,1)}{y_{u}})\leq f(1-k),\,\lambda_{d_{2}}(1+\frac{\beta (1,1)}{y_{v}})\leq g(1-k^{\prime}).$$
Then, we normalize $\phi _{d_{0}}$, $\phi _{d_{1}}$ and $\phi _{d_{2}}$
by requiring that
 $\|\phi_{d_{0}}\|<\min(k,k^{\prime})<1$ and
%
\begin{equation}
\label{norm}
 \phi _{d_{2}}(x)\leq \phi _{d_{1}}(x)\leq \frac{\lambda_{d_{0}}}{f(1)+g(1)}\phi _{d_{0}}(x).
\,\,0\leq x\leq 1.
\end{equation}
Define the operator $G(S,u,v)=(G_{1}(S,u,v),G_{2}(S,u,v),G_{3}(S,u,v))$ on $\Sigma$ by \begin{equation}
\label{eqp110}
\begin{array}{llll}
&G_{1} (S,u,v)(x)=\displaystyle \int _{0}^{1}\exp (
\frac{\min (x,t)-t}{d_{0}})(f(1-S)u+g(1-S)v)dt,
\\
&G_{2} (S,u,v)(x)=\displaystyle \int _{0}^{1}\exp (
\frac{\min (x,t)-t}{d_{1}})[f(1-S)u+\beta (u,v)v-
\frac{1}{y_{u}}\alpha (u,v)u]dt,
\\
&G_{3} (S,u,v)(x)=\displaystyle \int _{0}^{1}\exp (
\frac{\min (x,t)-t}{d_{2}})[g(1-S)v+
\alpha (u,v)u-\frac{1}{y_{v}}\beta (u,v)v]dt.
\end{array}
%
\end{equation}
Then for all $(S,u,v)\in \Sigma,$ we have\\
$$0\leq G_{1} (S,u,v)(x)\leq \displaystyle \int _{0}^{1}\exp (
\frac{\min (x,t)-t}{d_{0}})(f(1)\phi_{d_{1}}+g(1)\phi_{d_{2}})dt.$$
Also, thanks to (\ref{norm}), we have $$0\leq G_{1} (S,u,v)(x)\leq \displaystyle \int _{0}^{1}\exp (
\frac{\min (x,t)-t}{d_{0}})(f(1)+g(1))\phi_{d_{1}}dt.$$
So, we obtain $$0\leq G_{1} (S,u,v)(x)\leq \displaystyle \int _{0}^{1}\exp (
\frac{\min (x,t)-t}{d_{0}})\lambda_{d_{0}}\phi_{d_{0}}dt,$$
and this implies $0\leq G_{1} (S,u,v)\leq \phi_{d_{0}}.$\\
Next, since $\|\phi_{d_{0}}\|<k<1$, we see that for all $(S,u,v)\in \Sigma$ we have
$$G_{2} (S,u,v)(x)\geq \displaystyle\int _{0}^{1}
\exp (
\frac{\min (x,t)-t}{d_{1}})(f(1-k)\frac{\phi_{d_{1}}}{\lambda_{d_{1}}}-\frac{1}{y_{u}}\alpha(u,v)\phi_{d_{1}}+\beta (u,v)\frac{\phi_{d_{2}}}{\lambda_{d_{2}}})dt.$$
So using that $\theta\phi_{d_{1}}\leq\phi_{d_{2}}$, then by the monotonicity of $\alpha(\dot),$
 we have $$G_{2} (S,u,v)(x)\geq  \displaystyle\int _{0}^{1}\exp (
\frac{\min (x,t)-t}{d_{1}})(\frac{\phi_{d_{1}}}{\lambda_{d_{1}}}f(1-k)-\frac{1}{y_{u}}\alpha(1,1)\phi_{d_{1}}+\beta (u,v)\frac{\theta\phi_{d_{1}}}{\lambda_{d_{2}}})dt.$$
Then using that $\rho\leq\beta(u,v)$ for all $(S,u,v)\in \Sigma,$ we obtain $$G_{2} (S,u,v)(x)\geq  \displaystyle (f(1-k)\frac{1}{\lambda_{d_{1}}}-\frac{1}{y_{u}}\alpha(1,1)+\theta\rho \frac{1}{\lambda_{d_{2}}})) \displaystyle\int _{0}^{1}\exp (
\frac{\min (x,t)-t}{d_{1}})\phi_{d_{1}}.$$
Consequently, we have $$G_{2} (S,u,v)(x)\geq  \displaystyle (f(1-k)\frac{1}{\lambda_{d_{1}}}-\frac{1}{y_{u}}\alpha(1,1)+\theta\rho \frac{1}{\lambda_{d_{2}}})) \frac{\phi_{d_{1}}}{\lambda_{d_{1}}}.$$
Since, $ \frac{\theta\rho}{\lambda_{d_{2}}}>0$, we have $1+\frac{\alpha(1,1)}{y_{u}}\leq \frac{\theta\rho}{\lambda_{d_{2}}}+\frac{f(1-k)}{\lambda_{d_{1}}}$, and we obtain that $$G_{2} (S,u,v)(x)\geq \frac{\phi_{d_{1}}}{\lambda_{d_{1}}}.$$
Using again (\ref{norm}), we get for all $(S,u,v)\in \Sigma,$
$$G_{2} (S,u,v)(x)\leq \displaystyle\int _{0}^{1}
\exp (
\frac{\min (x,t)-t}{d_{1}})(f(1)\phi_{d_{1}}-\frac{1}{y_{u}}\theta\frac{\phi_{d_{1}}}{\lambda_{d_{1}}}+\beta (1,1)\phi_{d_{1}}dt.$$
So, we have $$G_{2} (S,u,v)(x)\leq \displaystyle(f(1)\frac{1}{\lambda_{d_{1}}}-\frac{1}{y_{u}}\theta\frac{1}{\lambda^{2}_{d_{1}}}+\beta (1,1)\frac{1}{\lambda_{d_{1}}})\phi_{d_{1}}(x).$$
Using, the assumption (\ref{A2}), we conclude $G_{2} (S,u,v)(x)\leq \displaystyle\phi_{d_{1}}(x).$

Next, since
$\phi _{d_{0}}<\min(k,k^{\prime})$, we find that for all $(S,u,v)\in \Sigma,$
$$G_{3} (S,u,v)(x)\geq \displaystyle\int _{0}^{1}
\exp (
\frac{\min (x,t)-t}{d_{2}})(g(1-k^{\prime})\frac{\phi_{d_{2}}}{\lambda_{d_{2}}}-\frac{1}{y_{v}}\beta(u,v)\phi_{d_{2}}+\alpha (u,v)\frac{\phi_{d_{1}}}{\lambda_{d_{1}}})dt.$$
So using again (\ref{norm}), and the monotonicity of $\beta(\cdot)$, we have $$G_{3} (S,u,v)(x)\geq  \displaystyle\int _{0}^{1}\exp (
\frac{\min (x,t)-t}{d_{2}})(\frac{\phi_{d_{2}}}{\lambda_{d_{2}}}g(1-k^{\prime})-\frac{1}{y_{v}}\beta(1,1)\phi_{d_{2}}+\alpha (u,v)\frac{\phi_{d_{2}}}{\lambda_{d_{1}}})dt.$$
Then $$G_{3} (S,u,v)(x)\geq  \displaystyle (g(1-k^{\prime})\frac{1}{\lambda_{d_{2}}}-\frac{1}{y_{v}}\beta(1,1)+\theta \frac{1}{\lambda_{d_{1}}}))\frac{\phi_{d_{2}}}{\lambda_{d_{2}}}.$$
Since, we have $1+\frac{\beta(1,1)}{y_{v}}\leq \frac{\theta}{\lambda_{d_{1}}}+\frac{g(1-k^{\prime})}{\lambda_{d_{2}}}$, we obtain that $$G_{3} (S,u,v)(x)\geq\displaystyle \frac{\phi_{d_{2}}}{\lambda_{d_{2}}}.$$
Using again the monotonicity of $\alpha(\cdot)$, we get for all $(S,u,v)\in \Sigma,$
$$G_{3} (S,u,v)(x)\leq \displaystyle\int _{0}^{1}
\exp (
\frac{\min (x,t)-t}{d_{2}})(g(1)\phi_{d_{2}}-\frac{1}{y_{v}}\rho\frac{\phi_{d_{2}}}{\lambda_{d_{2}}}+\alpha (1,1)\phi_{d_{1}}dt.$$
So, we have $$G_{3} (S,u,v)(x)\leq \displaystyle(g(1)\frac{1}{\lambda_{d_{2}}}-\frac{1}{y_{v}}\rho\frac{1}{\lambda^{2}_{d_{2}}}+\alpha (1,1)\frac{1}{\lambda_{d_{2}}\theta})\phi_{d_{2}}(x).$$
Using, the assumption (\ref{A4}) leads to $G_{3} (S,u,v)(x)\leq \displaystyle\phi_{d_{2}}(x).$

 Then, from well-known arguments in \cite{A},
$G$ is a compact continuous operator on $\Sigma$, and we have that $G(\Sigma)\subset \Sigma.$ So we deduce using the Schauder fixed point theorem that there exists $(S,u,v)\in \Sigma$ satisfying $T(S,u,v)=(S,u,v)$. That is, $G_{1} (S,u,v)=S$ and $G_{2} (S,u,v)=u$ and $G_{3} (S,u,v)=v$. This implies $(S,u,v)$ is a non trivial solution of system (\ref{SS1})-(\ref{SS1b}).\\

\end{proof}

\section{Numerical simulations}\label{s6} In this section, we illustrate the effect of the coefficients of diffusion on the coexistence
of microbial species as well as their competitive behaviour through numerical
simulations. Let us consider the flocculation-deflocculation rates defined by:

$$\alpha(u,v) =(u
+v)v,\,\,\beta(u,v) =(1+v)(u+v),\,y_{u}=y_{v}=10^{-1},$$
  with initial conditions
\begin{equation*}
S(x,0)=0.1,\,\, u(x,0)=1,\,\,v(x,0)=1,\,\,x\in[0,1].
\end{equation*}

 Let us recall the stability result of the washout steady state $E_{0}=(1,0,0)$ established in \cite{ZA},
\begin{theorem}(\cite{ZA})
\label{swss}
Assume that the functions $f,g $ are positive and continuously differentiable such that
$f (0)=g(0)=0$ and
$\alpha ,\beta $ are positive and continuously differentiable such that
$\alpha (0,0)=0$ or $\beta (0,0)=0$. Then, if the first eigenvalues
$-\lambda _{d_{1}},\,\,-\lambda _{d_{2}}$ associating
$(P_{\lambda })$ with $d=d_{1}$, $d=d_{2},$ respectively, satisfy
\begin{equation*}
\lambda _{d_{1}}>f(1)-\frac{1}{y_{u}}\alpha (0,0)\text{ and }\lambda _{d_{2}}>g(1)-
\frac{1}{y_{v}}\beta (0,0),
\end{equation*}
then the washout steady state $E_{0}$ is uniformly asymptotically stable.%

If
\begin{equation*}
\lambda _{d_{1}}<f(1)-\frac{1}{y_{u}}\alpha (0,0)\text{ or }\lambda _{d_{2}}<g(1)-
\frac{1}{y_{v}}\beta (0,0),
\end{equation*}
then the washout steady state is unstable.
\end{theorem}

Indeed, the biological meaning of $\frac{1}{\lambda_{d}}$ is interpreted as the mean residence time of microbe in the reactor (see \cite{BL}), and the function $d\mapsto \lambda_{d}$ is strictly decreasing in $d\in(0,+\infty)$ and satisfies $$\displaystyle\lim_{d\rightarrow 0^{+}}\lambda_{d}=+\infty,$$  $$\displaystyle\lim_{d\rightarrow +\infty}\lambda_{d}=1,$$ and if $d<\frac{1}{2\pi}$ then $$\frac{1}{4d}+\pi^{2}\frac{d}{4}<\lambda_{d}<\frac{1}{4d}+\pi^{2}d.$$
Let us define the reproductive number of isolated and attached bacteria in the flocculation model by $$R_{u}=(f(\gamma_{S})-\frac{1}{y_{u}}\alpha (0,0))\lambda_{d_{1}}^{-1},$$
$$R_{v}=(g(\gamma_{S})-\frac{1}{y_{u}}\beta (0,0))\lambda_{d_{2}}^{-1},$$
as a product of the net growth rate per unit biomass and the mean residence time of a microbe with random motility $d$.\\
Our simulations convince us that $R_{u}<1$ and $R_{v}<1$ lead to the washout and $R_{u}>1$ or $R_{v}>1$ leads to a persistence of microorganisms. While, coexistence can occurs if $R_{u}>1$ and $R_{v}>1$. This leads to the conclusion that all solutions approach some steady state solution and the outcome of competition between isolated and attached bacteria can depend rather subtly on their respective random motility coefficients and their intrinsic growth rate at the nutrient concentration.\\

Indeed, in fig1, let $d_{0}=1,d_{1}=1,d_{2}=10$ and consider the per-capita growth rate of the isolated bacteria of Haldeine type function, $f(S)=\displaystyle\frac{3S}{1+S+S^{2}}$. Then, fig1 illustrates the convergence of the solution towards the
coexisting steady state.
 Whereas, if $f$ is the Monod type function, $f(S)=\displaystyle\frac{4S}{1+S}$, with the same random motility coefficients, we observe in fig2 that the solution converges towards the steady state of
extinction of the attached bacteria and the dominance of isolated bacteria in the medium and so the population with the smaller motility coefficient competitively can
exclude its more motile rival. This leads to the conclusion that the less motile and slower
growing organism can competitively exclude a more motile and faster growing competitor. Perhaps this is due to the fact that the rate of deflocculation exceeds the flocculation rate. But, bifurcation with respect to $d_{1}$ can occur, see fig3, which illustrates the convergence of the solution towards the steady state of extinction of isolated bacteria when $d_{0}=1$, $d_{1}=0.1 $ and $d_{2}=10$.\\

In fig3, we show that coexistence at equilibrium can occur even when one species has a higher intrinsic growth rate at the nutrient concentration, provided that the same species is sufficiently more motile than its rival. The results of numerical simulations in fig4 and fig5 illustrate
the convergence of the solution towards the washout or the coexisting steady
state, or to a semi-trivial steady-state solution corresponding to the extinction
of one of the two species.
In fig4, (d) and (e) show the convergence towards
the washout steady state for a small value of $d_{1}$ and $d_{2}$, whereas (f), (g), (h) and (i) illustrate the bifurcation from the constant state to a spatially dependent state with respect to the value of $d_{1},d_{2}$. In fig5, we observe that a more motile but slower growing organism can competitively exclude a less
motile but faster growing competitor.


In conclusion, the monotonicity properties of $\lambda_{d}$, is crucial to understand the effects of parameter
variation on the survival of a bacterial population, and, the phenomenon of coexistence or exclusion depend not only on the value of the couple $(d_{1},d_{2})$ but also on the growth rates and on the flocculation-deflocculation rates.
Indeed, consider the flocculation-deflocculation rates defined by:

$$\alpha(u,v) =(u
+v),\,\,\beta(u,v) =1,\,y_{u}=y_{v}=10^{-1},$$
  with initial conditions
\begin{equation*}
S(x,0)=0.1,\,\, u(x,0)=1,\,\,v(x,0)=1,\,\,x\in[0,1].
\end{equation*}
Here we describe
the results of a simulation, in fig6, in which the two organisms have Monod
uptake functions and the flocculation rate exceeds the deflocculation rate. We observe that the solution converges towards the steady state of
extinction of the attached bacteria. And so, the population with the smaller motility coefficient can
competitively exclude its more motile rival.

\begin{figure}[!hbtp]

\centering
\includegraphics[scale=0.5]{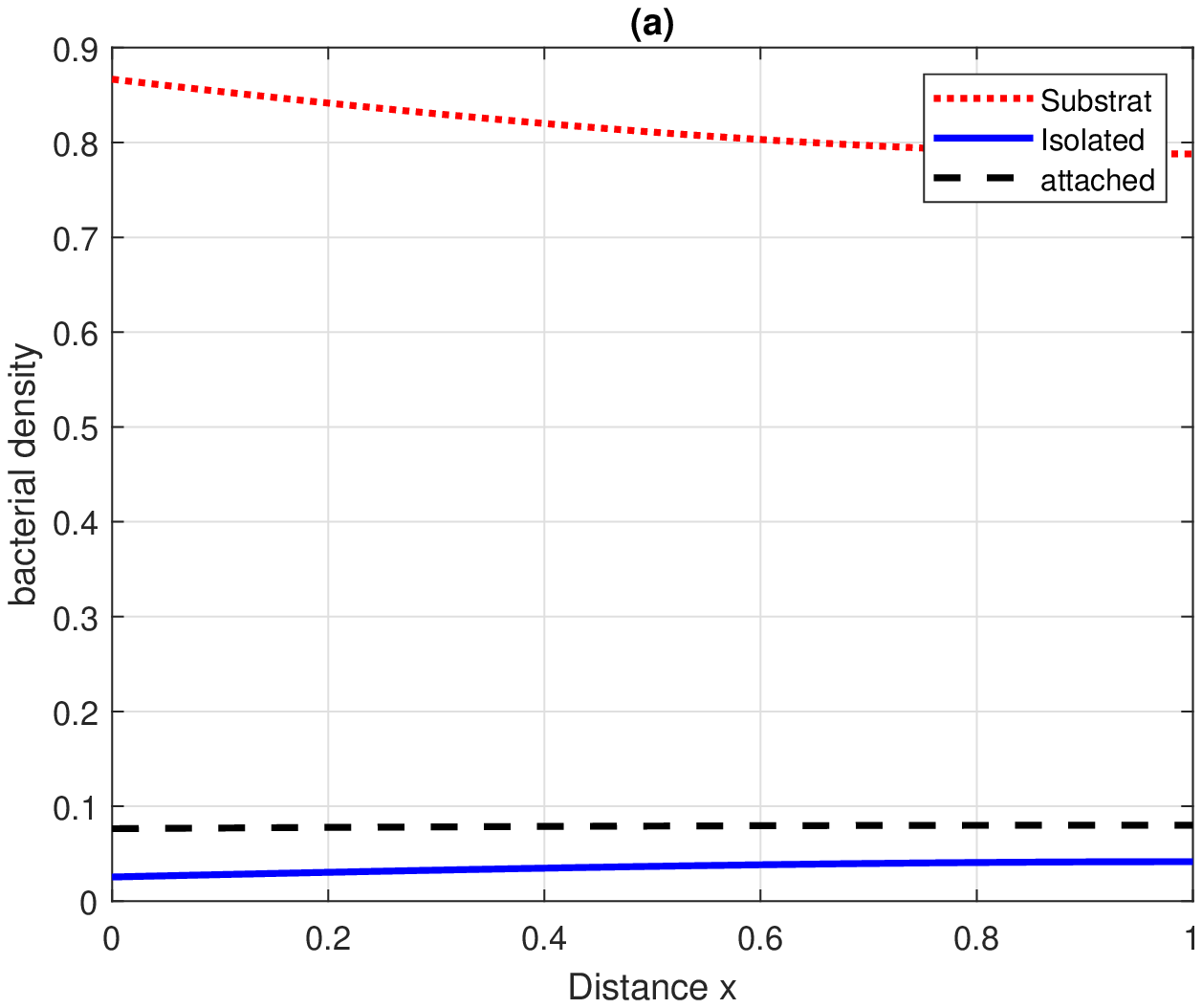}
\includegraphics[scale=0.5]{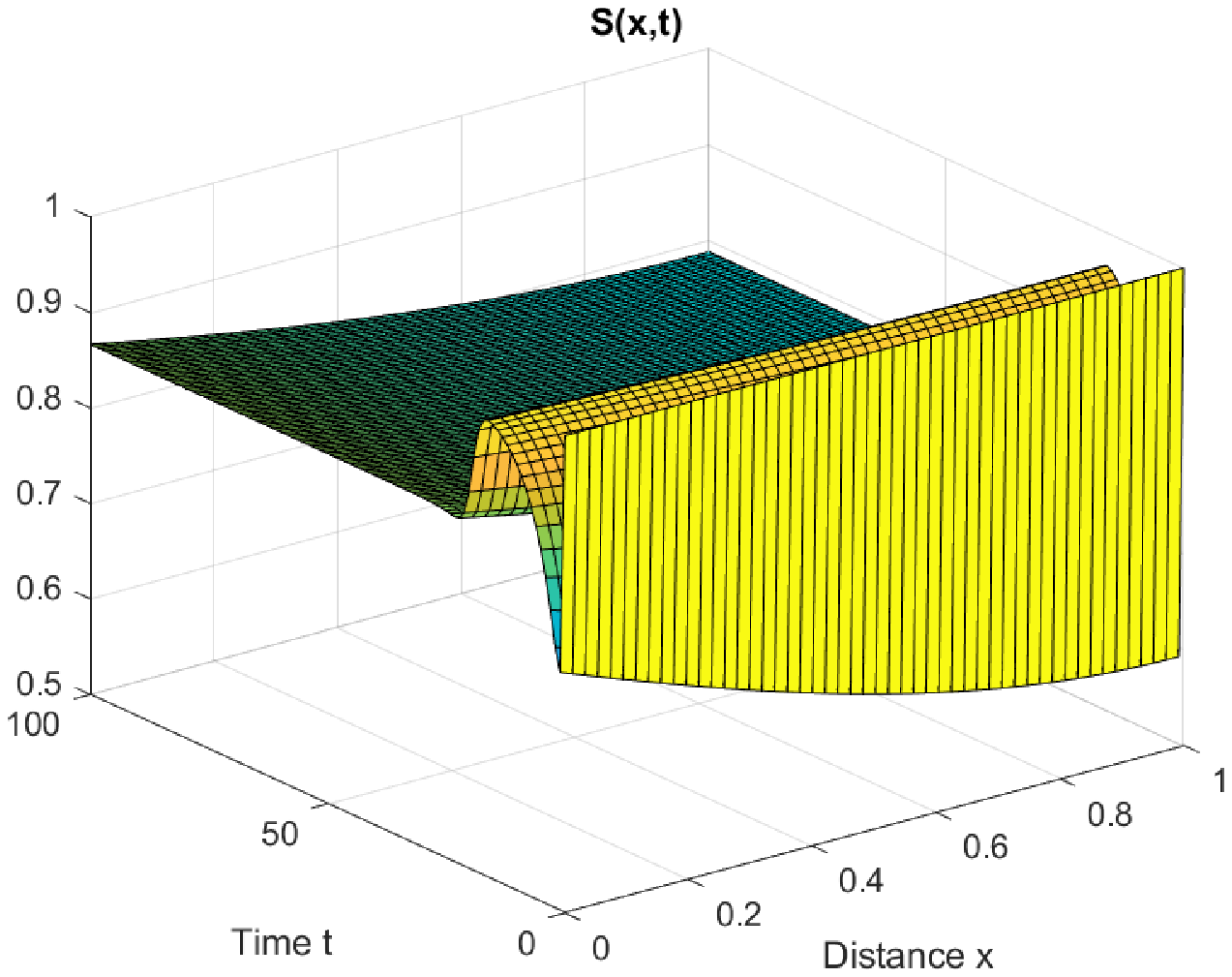}
\includegraphics[scale=0.5]{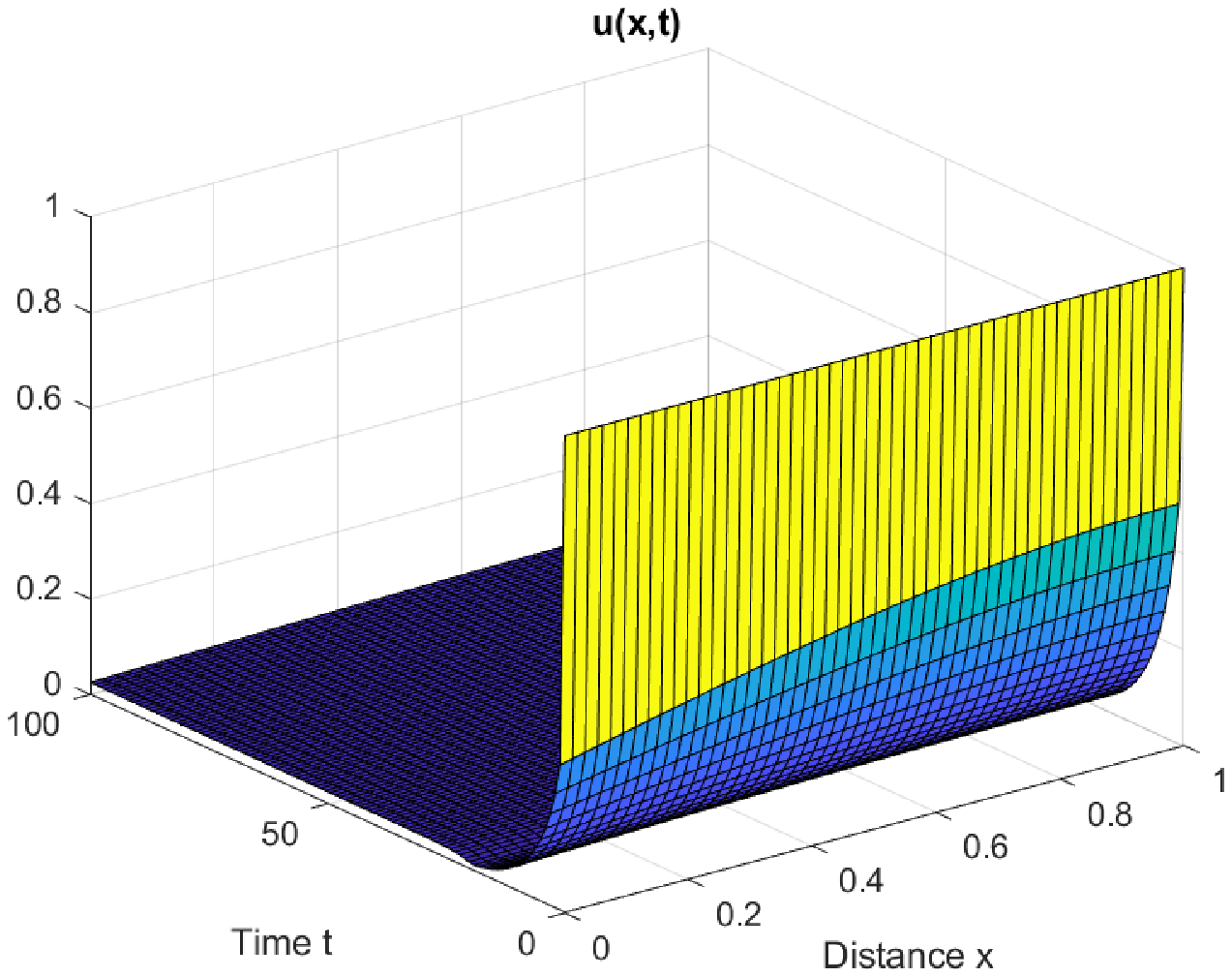}
\includegraphics[scale=0.5]{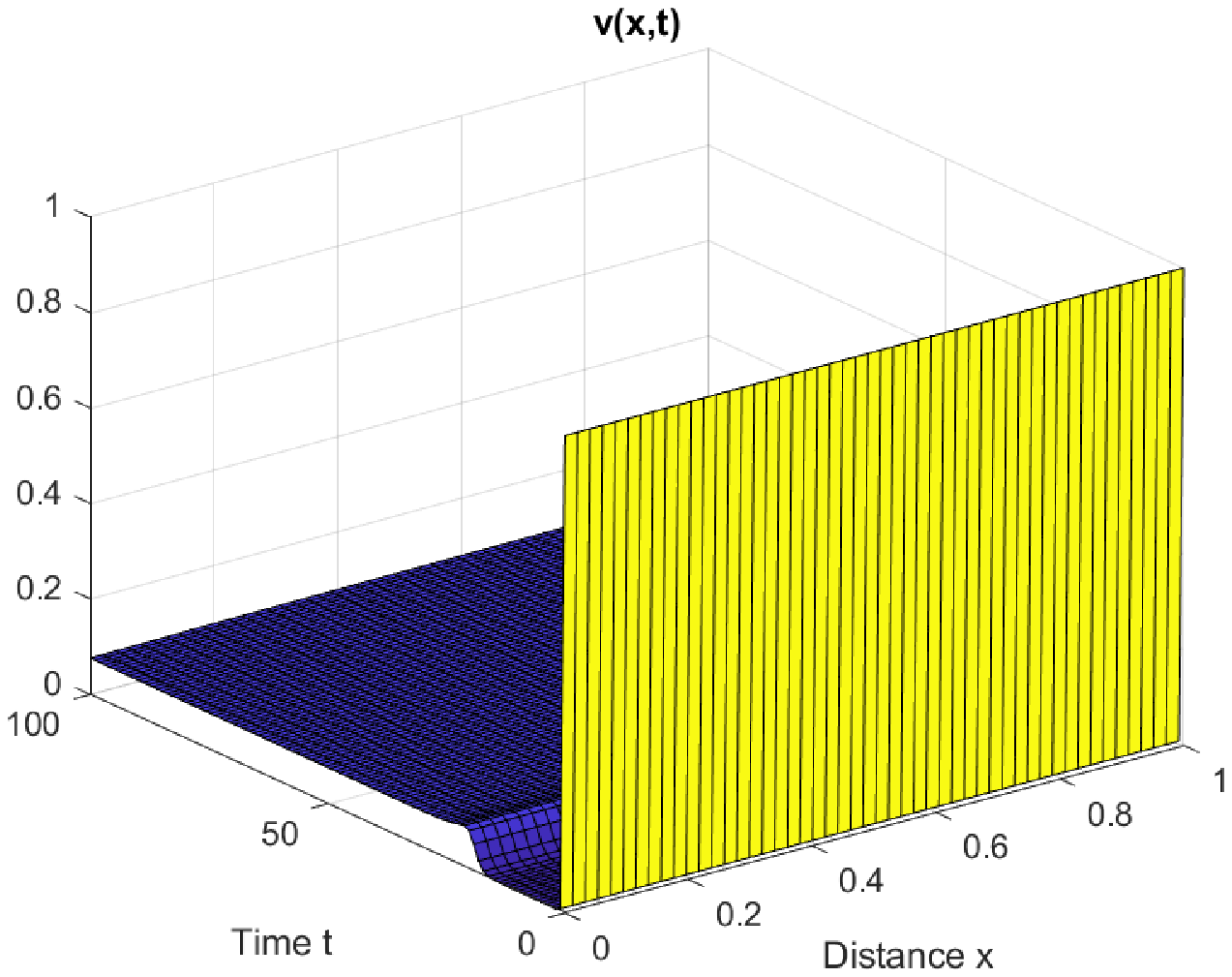}

\caption{The solution $(S, u, v)$ of system (2.1)-(2.3) as a function of two variables $x$ and $t$, in the case $(S_{0},u_{0},v_{0})=(1,0.1,0.1)$ $f(S)=\displaystyle\frac{3S}{1+S+S^{2}},\,\,g(S)=\displaystyle\frac{5S}{1+S},\,\alpha(u,v) =(u
+v)v,\,\,\beta(u,v) =(1+v)(u+v),$ $d_{0}=1,d_{1}=1,d_{2}=10$. Convergence towards the coexistence.}
\end{figure}
\begin{figure}[!hbtp]

\centering
\includegraphics[scale=0.5]{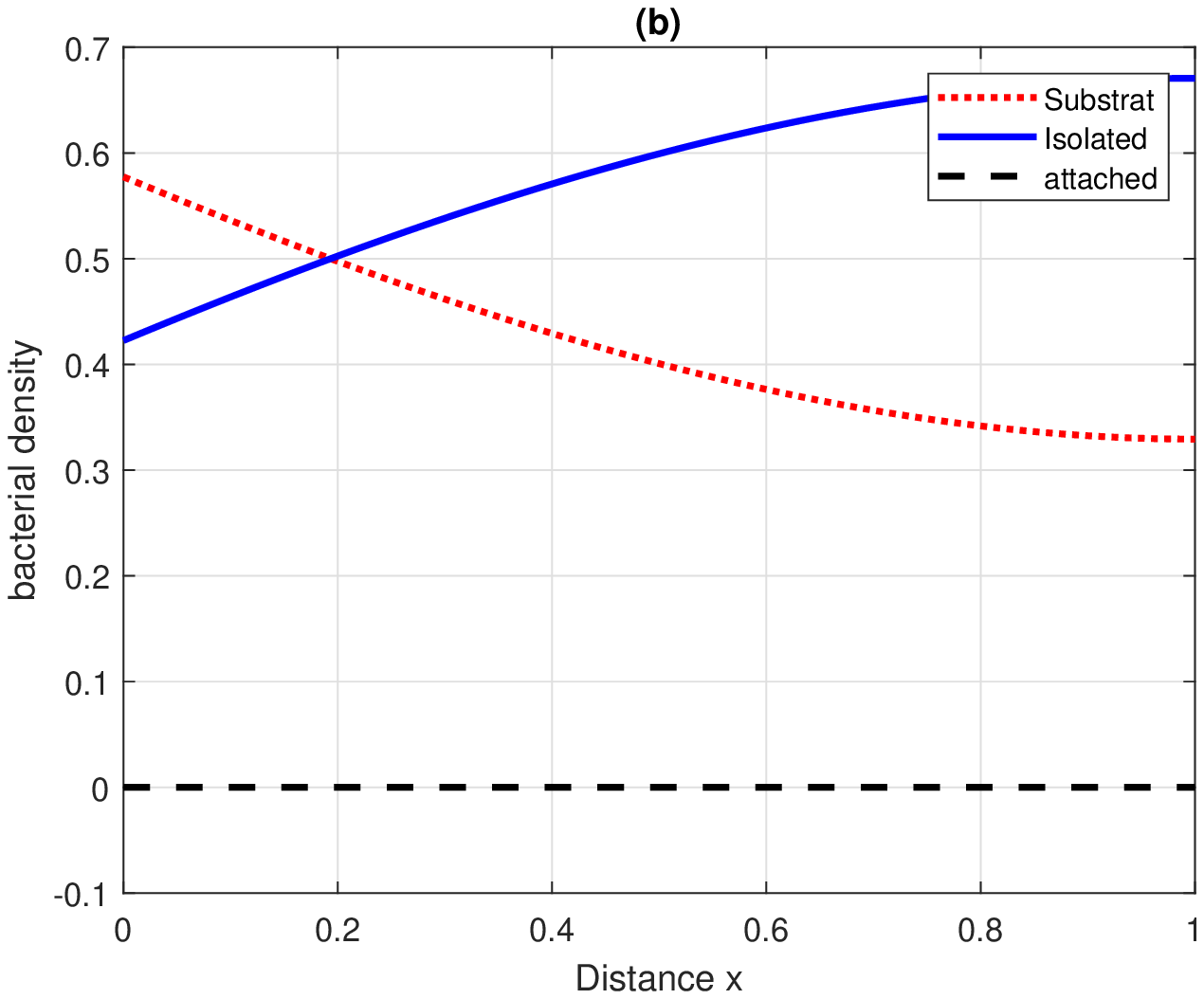}
\includegraphics[scale=0.5]{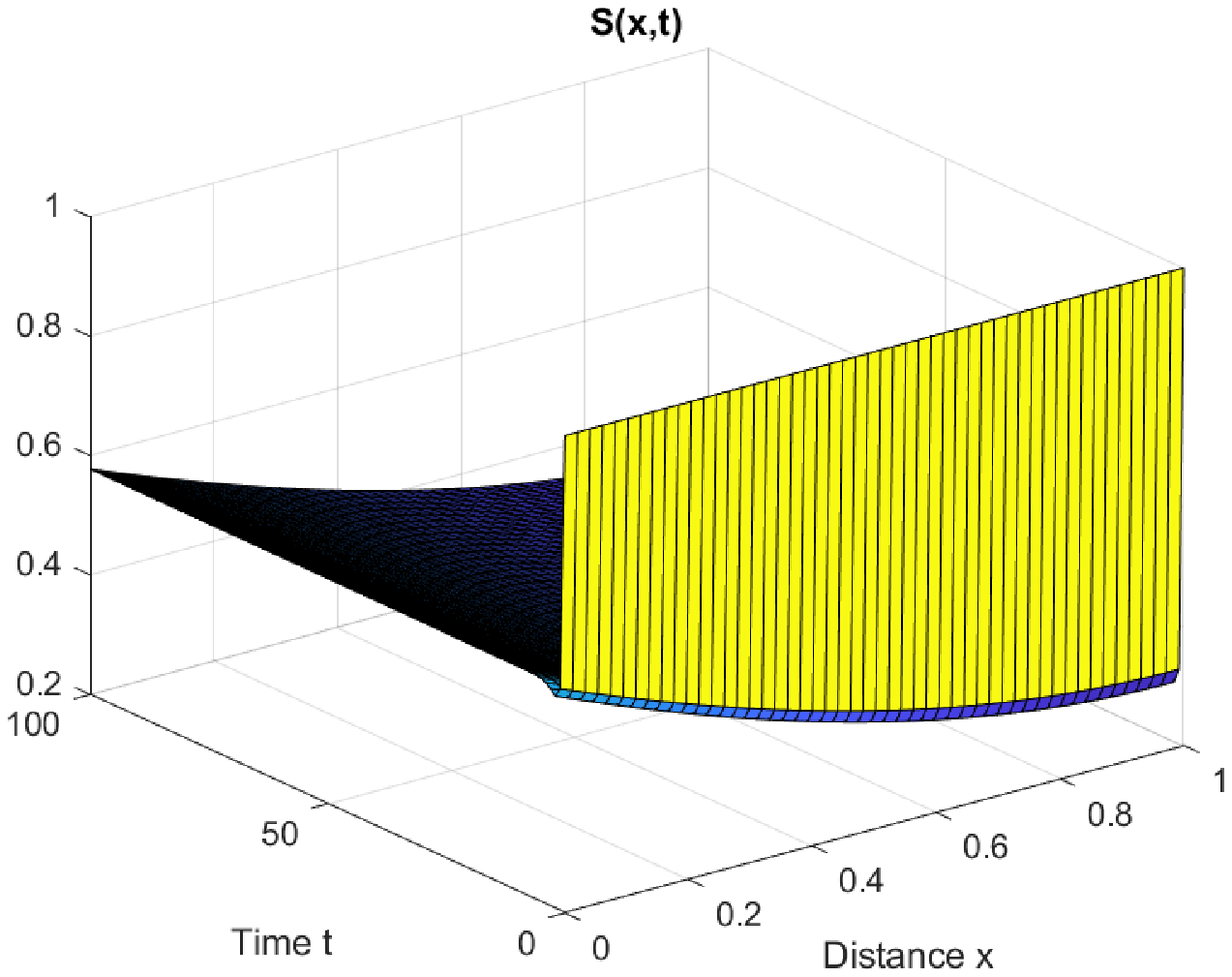}
\includegraphics[scale=0.5]{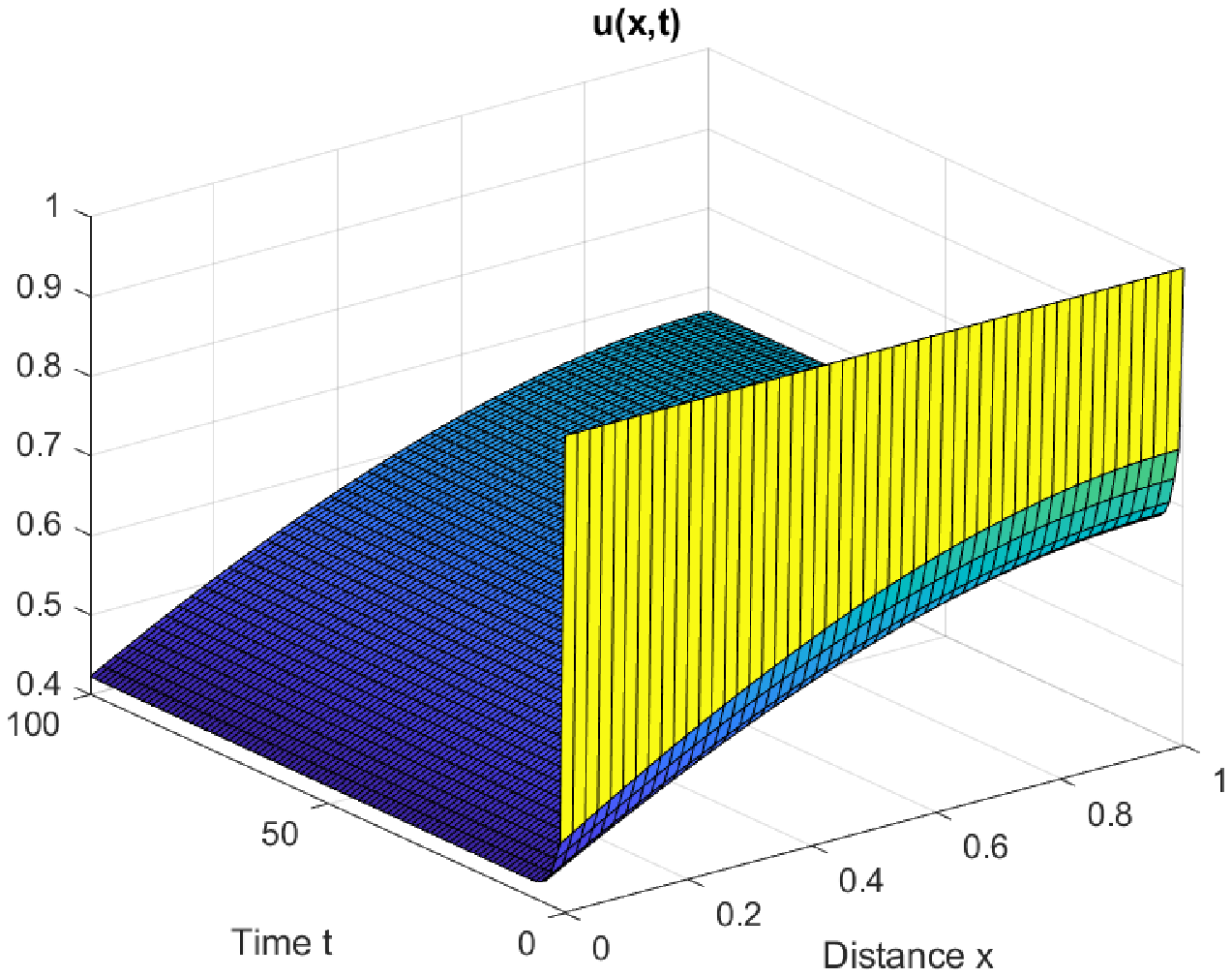}
\includegraphics[scale=0.5]{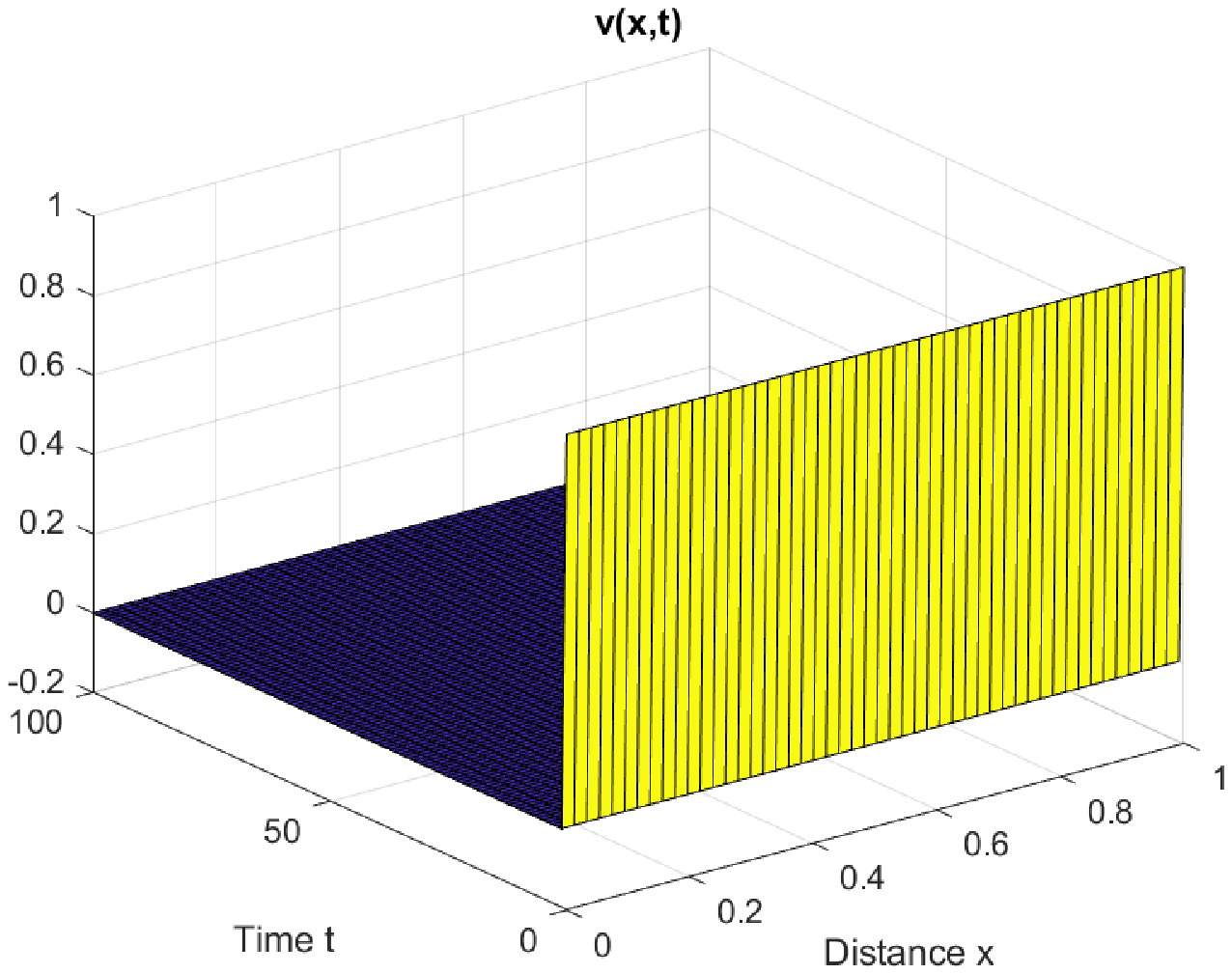}

\caption{The solution $(S, u, v)$ of system (2.1)-(2.3) as a function of two variables $x$ and $t$, in the case $(S_{0},u_{0},v_{0})=(1,0.1,0.1)$ $f(S)=\displaystyle\frac{4S}{1+S},\,\,g(S)=\displaystyle\frac{5S}{1+S},\,\alpha(u,v) =(u
+v)v,\,\,\beta(u,v) =(1+v)(u+v),$ $d_{0}=1,d_{1}=1,d_{2}=10$. Convergence towards the extinction of attached bacteria.}
\end{figure}
\begin{figure}[!hbtp]

\centering
\includegraphics[scale=0.5]{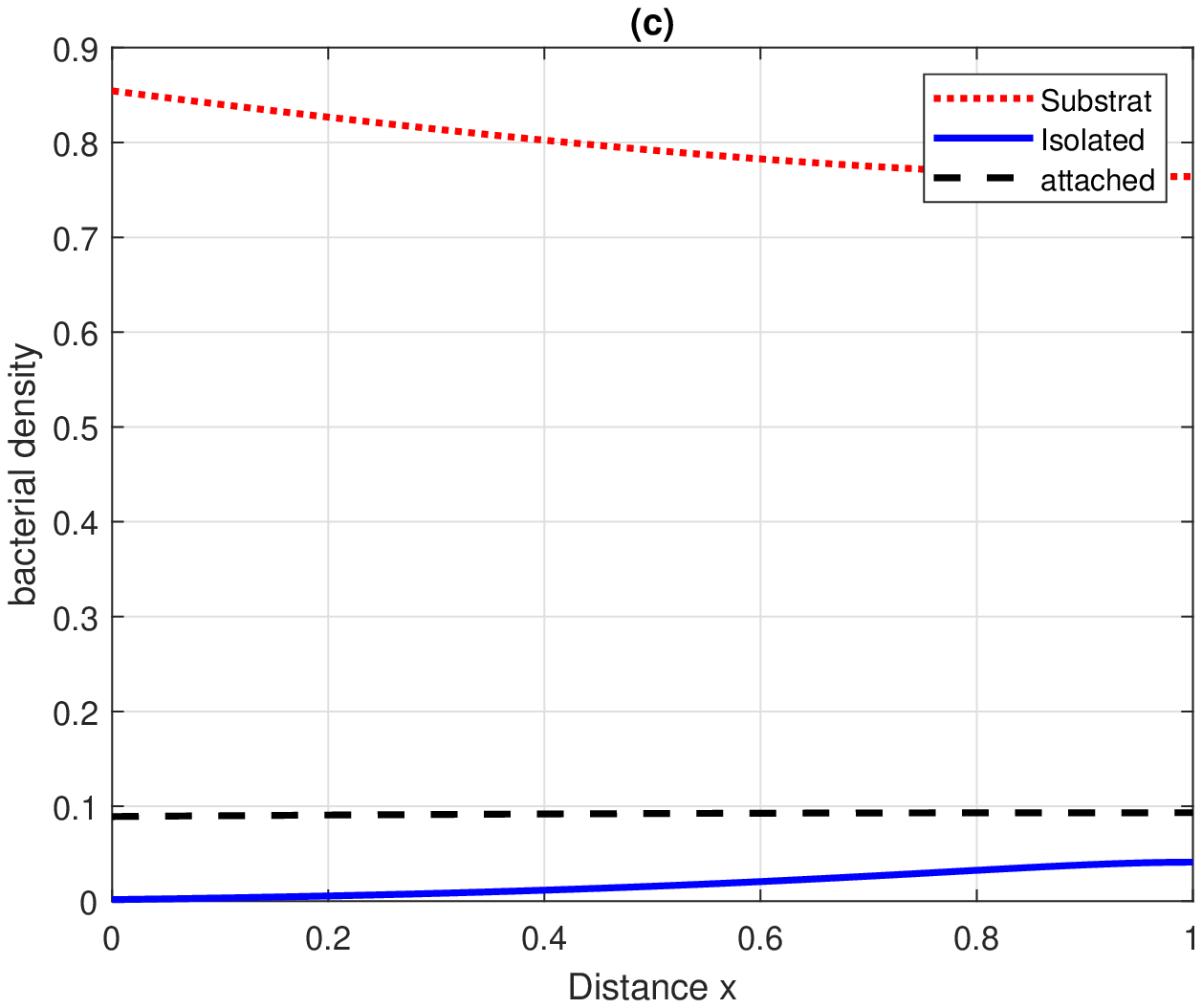}
\includegraphics[scale=0.5]{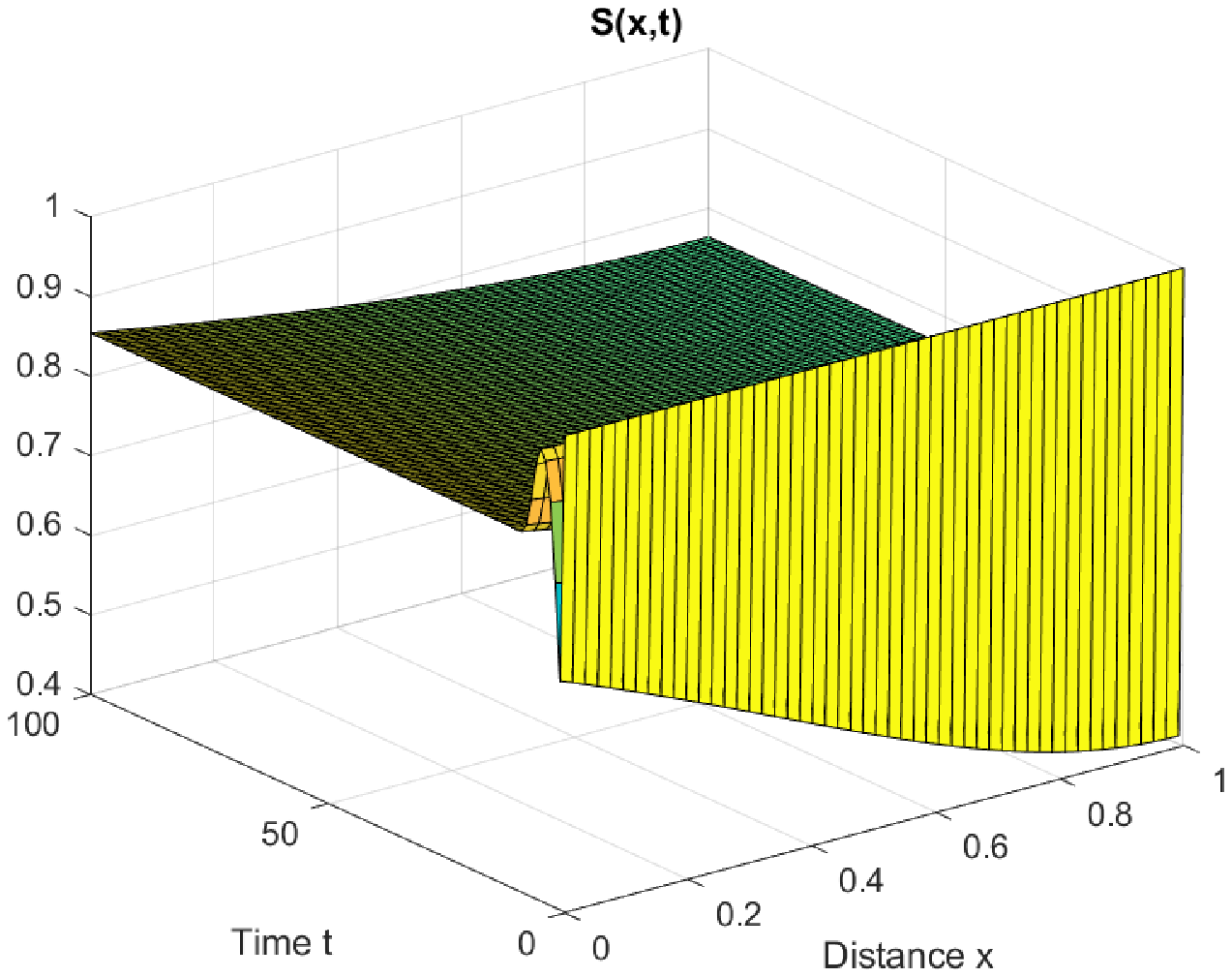}
\includegraphics[scale=0.5]{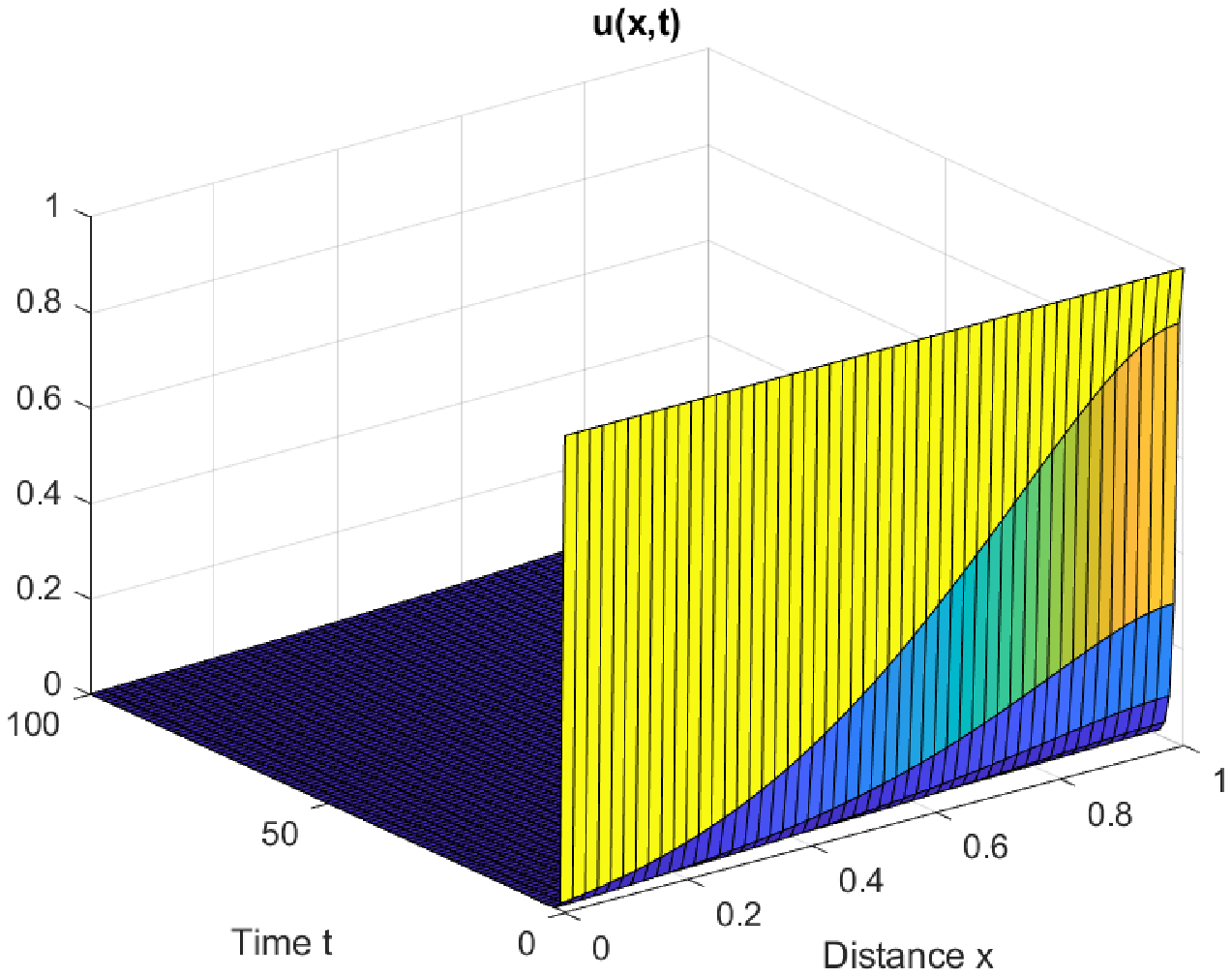}
\includegraphics[scale=0.5]{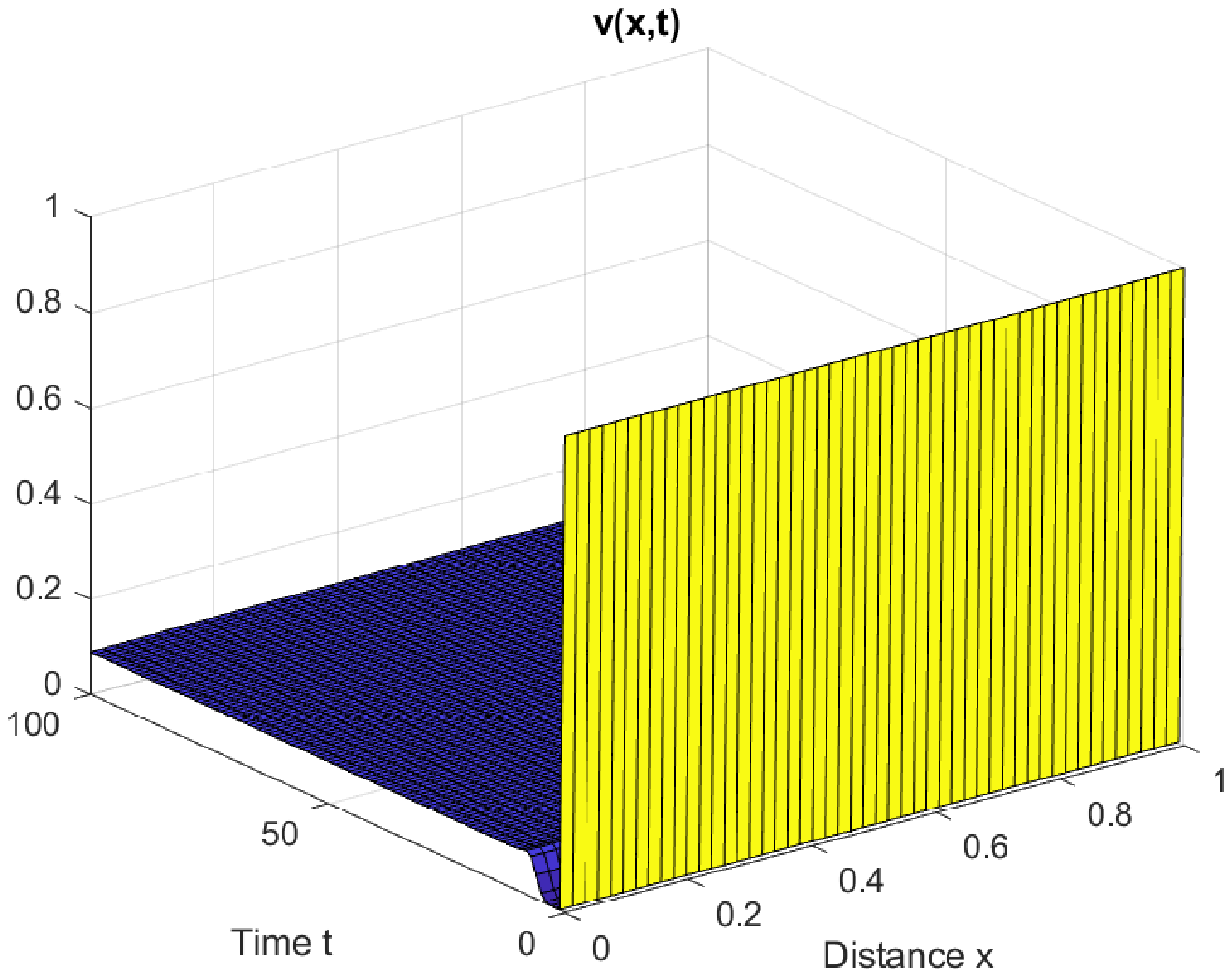}

\caption{The solution $(S, u, v)$ of system (2.1)-(2.3) as a function of two variables $x$ and $t$, in the case $(S_{0},u_{0},v_{0})=(1,0.1,0.1)$ $f(S)=\displaystyle\frac{4S}{1+S},\,\,g(S)=\displaystyle\frac{5S}{1+S},\,\alpha(u,v) =(u
+v)v,\,\,\beta(u,v) =(1+v)(u+v),$ $d_{0}=1,d_{1}=0.1,d_{2}=10$. Convergence towards the coexistence.}
\end{figure}
\begin{figure}[!hbtp]

\centering
\includegraphics[scale=0.5]{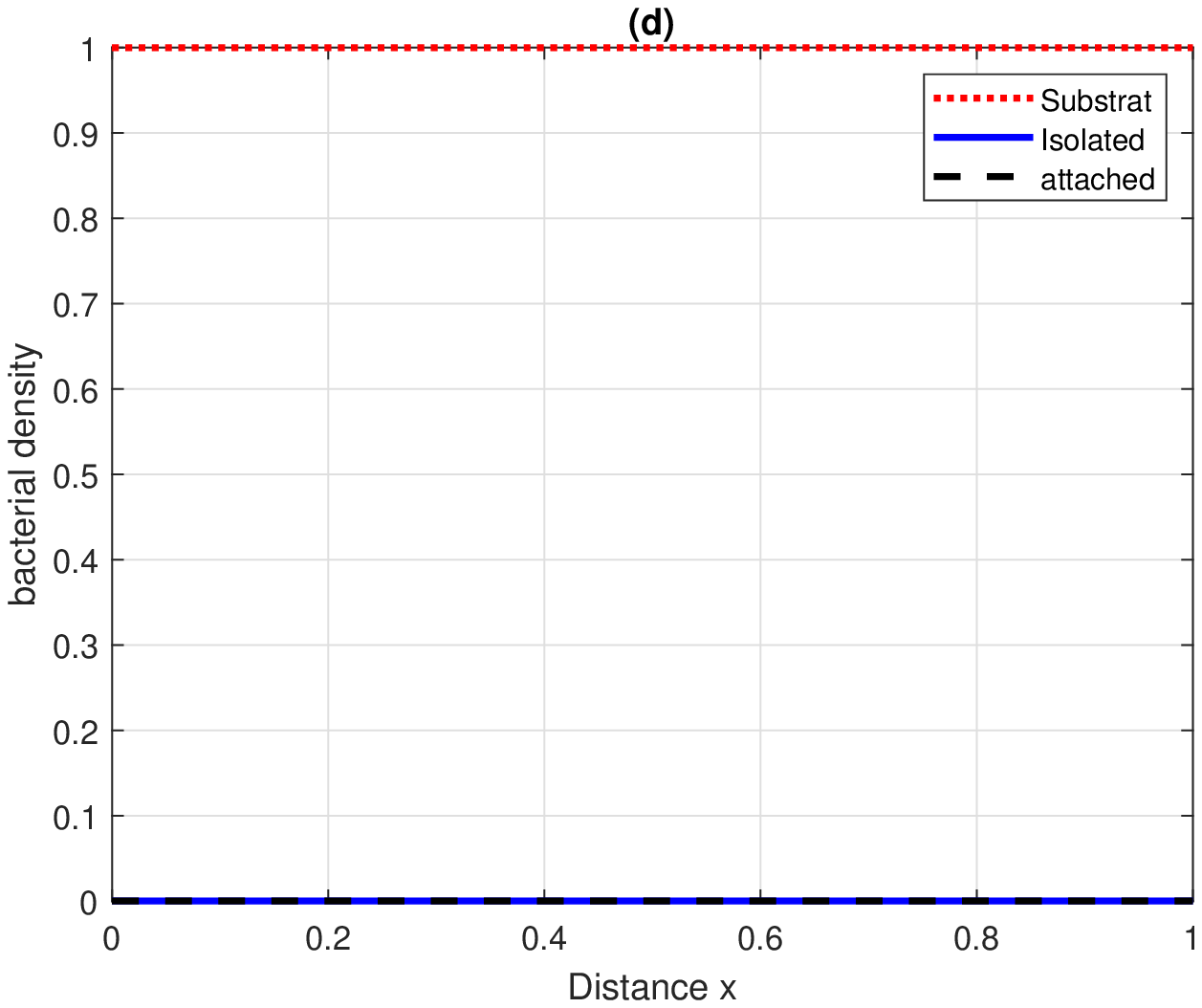}

\includegraphics[scale=0.5]{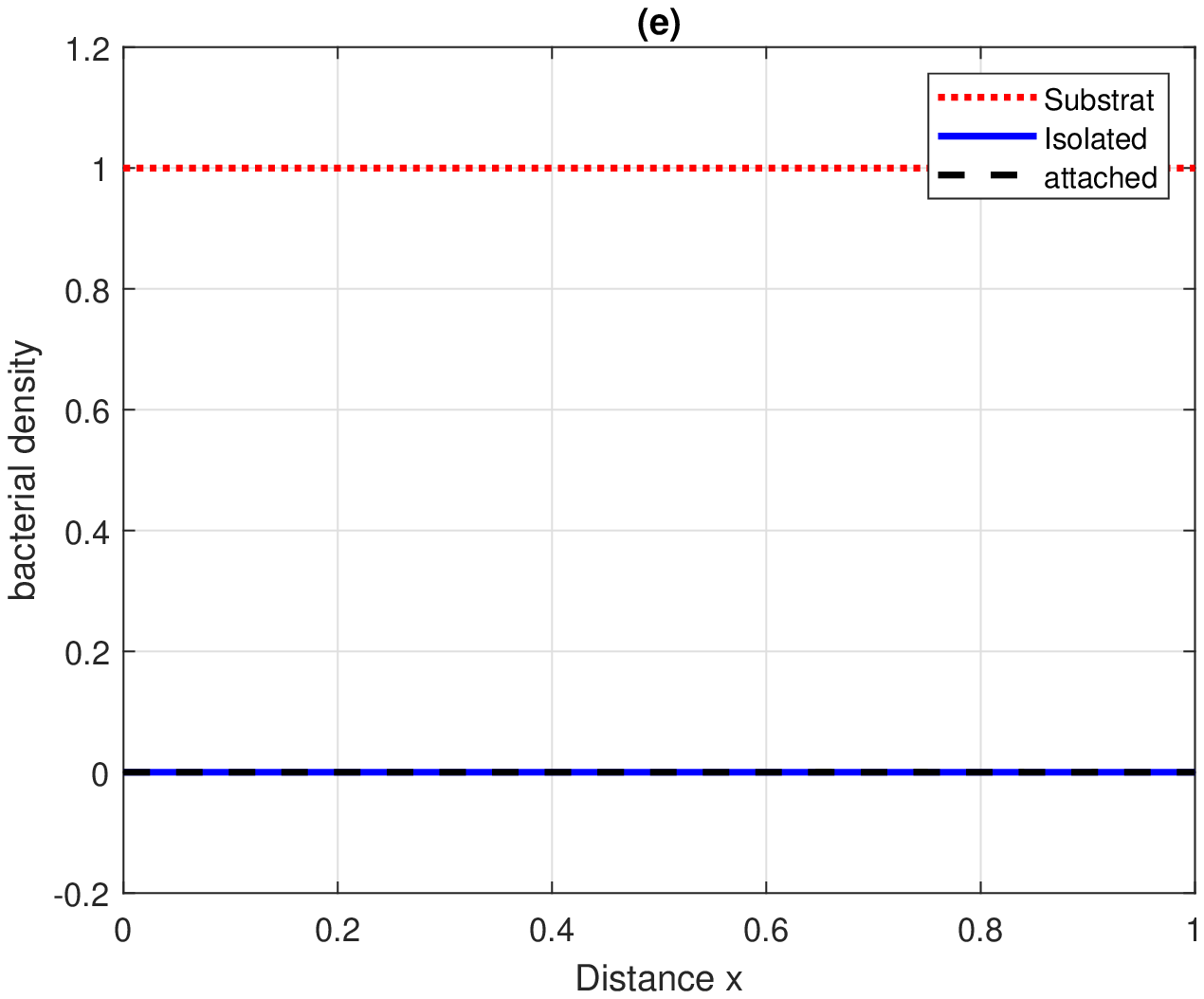}

\includegraphics[scale=0.5]{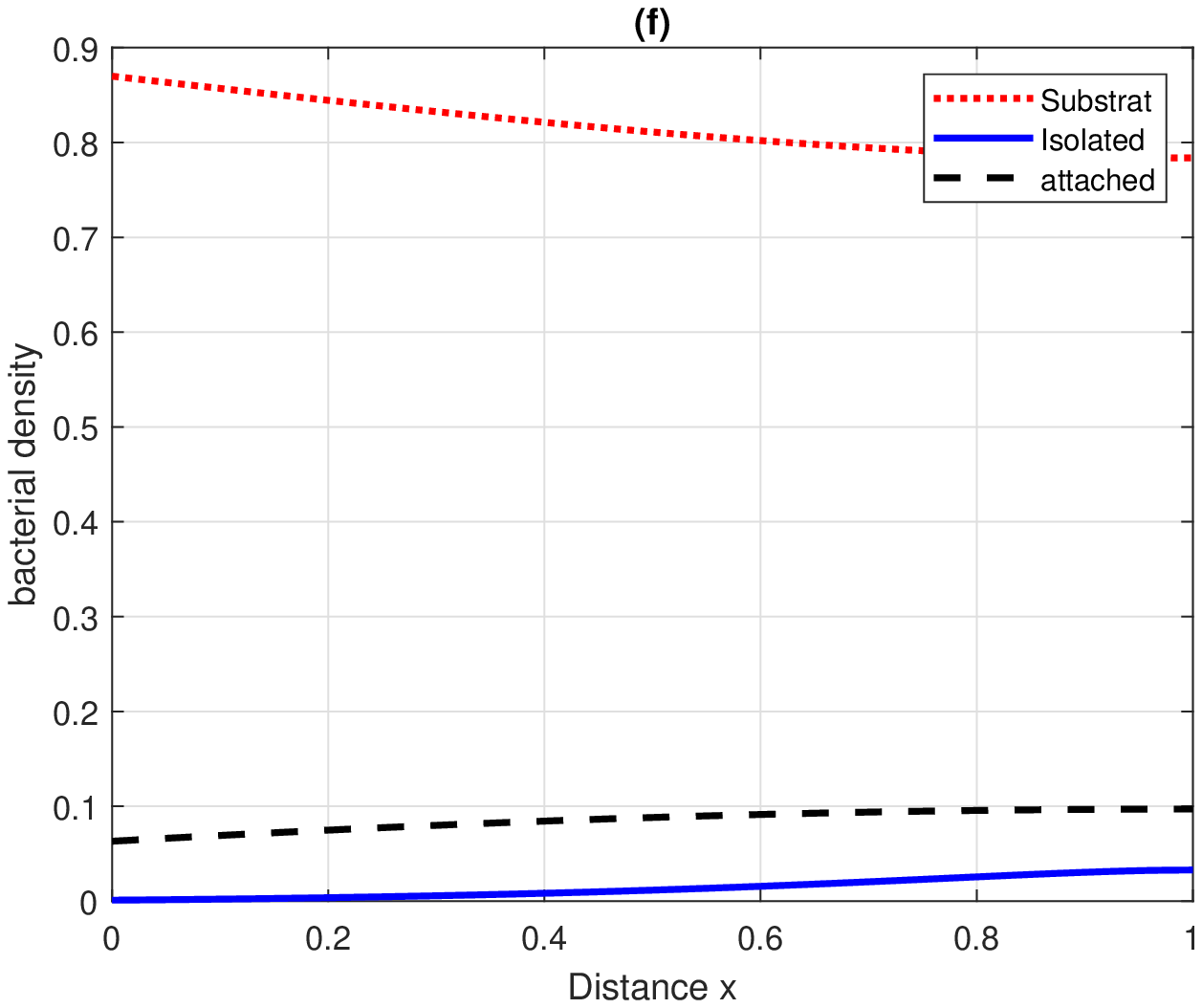}
\includegraphics[scale=0.5]{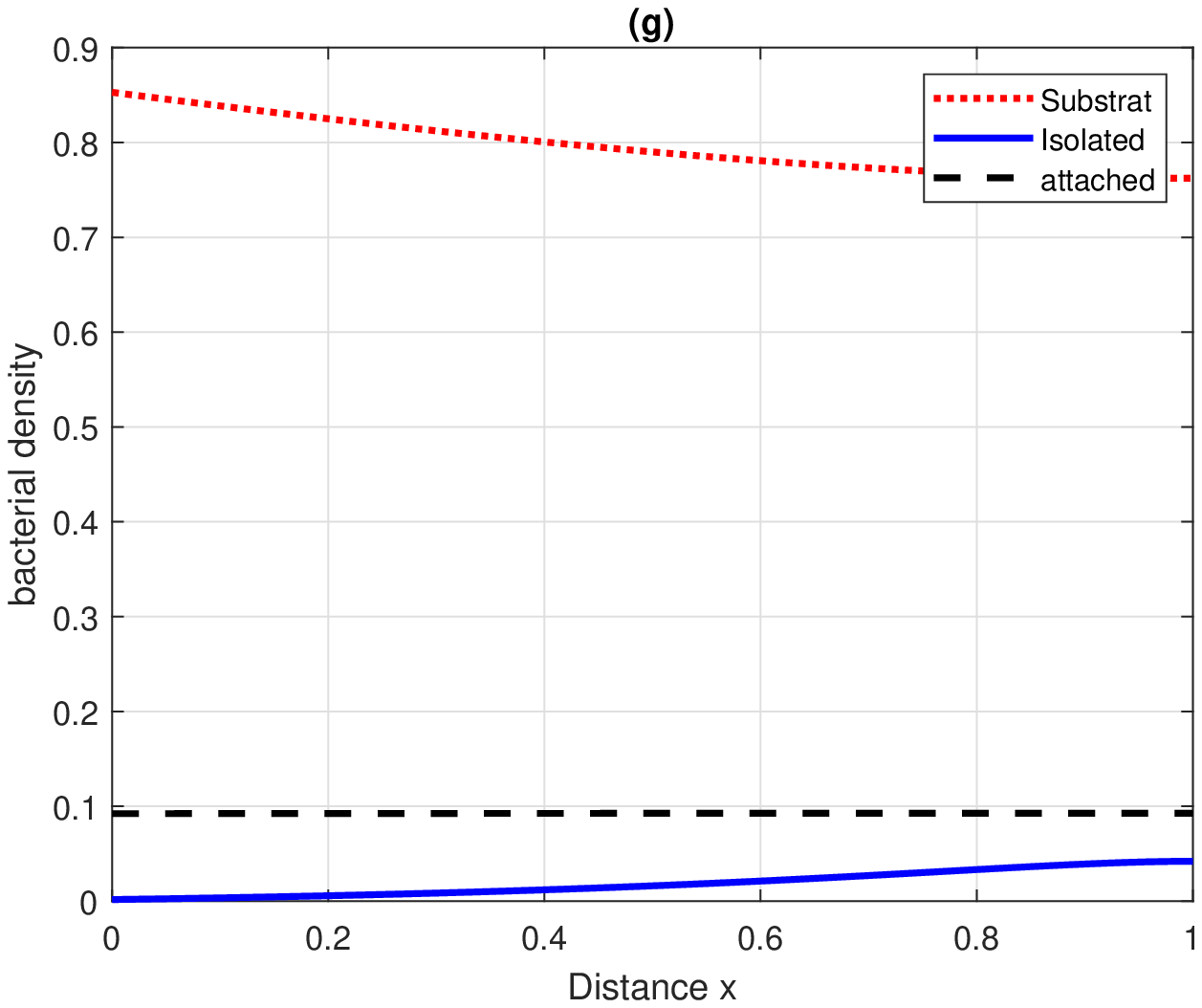}

\caption{The solution $(S, u, v)$ of system (2.1)-(2.3) as a function of two variables $x$ and $t=100$, in the case $f(S)=\displaystyle\frac{4S}{1+S},\,\,g(S)=\displaystyle\frac{5S}{1+S},$ $(d)$ $(d_{0}=1,d_{1}=0.1,d_{2}=0.1)$, $(e)$ $(d_{0}=1,d_{1}=0.1,d_{2}=0.001)$, $(f)$ $(d_{0}=1,d_{1}=0.1,d_{2}=1)$, $(g)$ $(d_{0}=1,d_{1}=0.1,d_{2}=100)$.}
\end{figure}

\begin{figure}[!hbtp]

\centering
\includegraphics[scale=0.5]{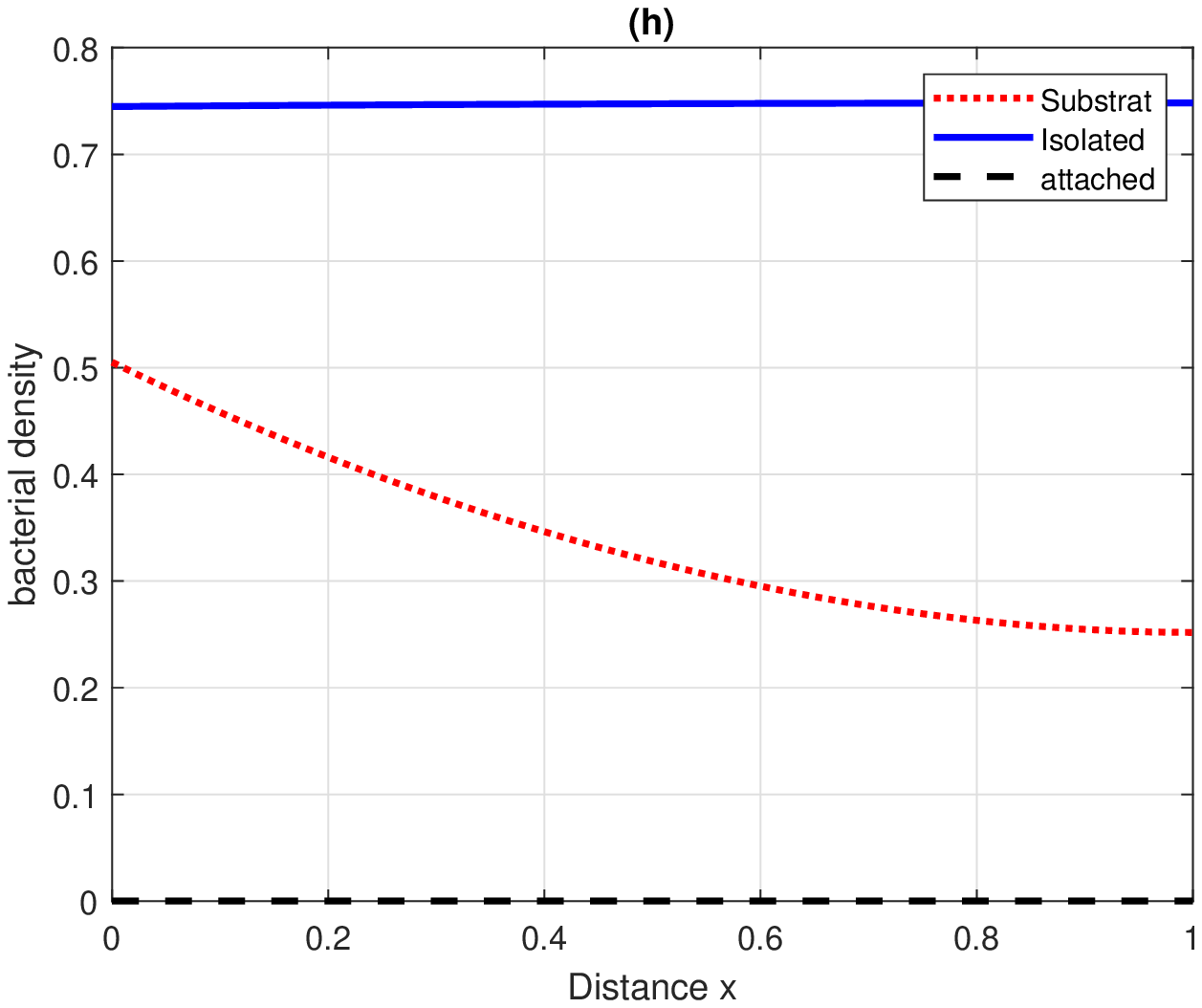}
\includegraphics[scale=0.5]{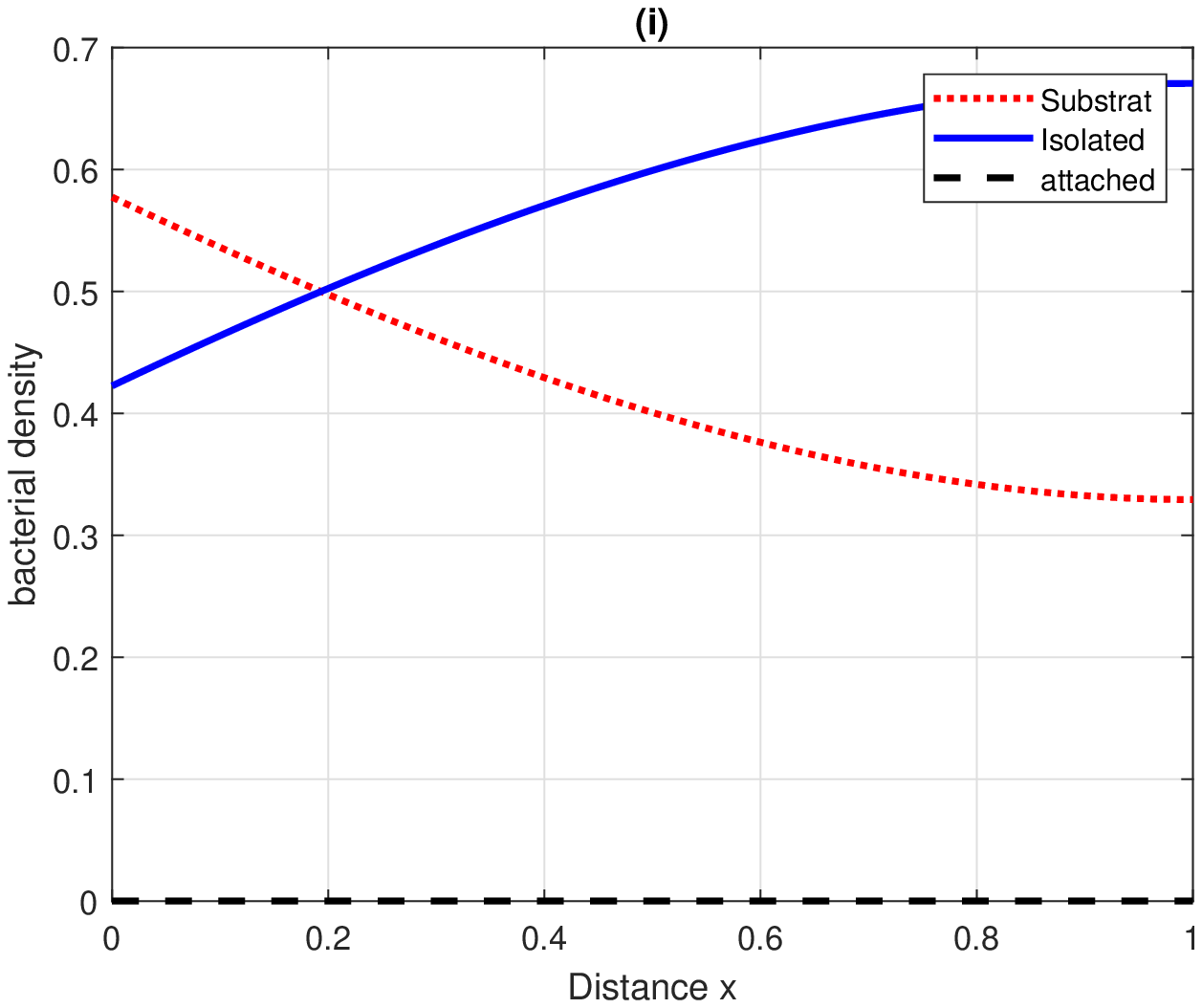}

\includegraphics[scale=0.5]{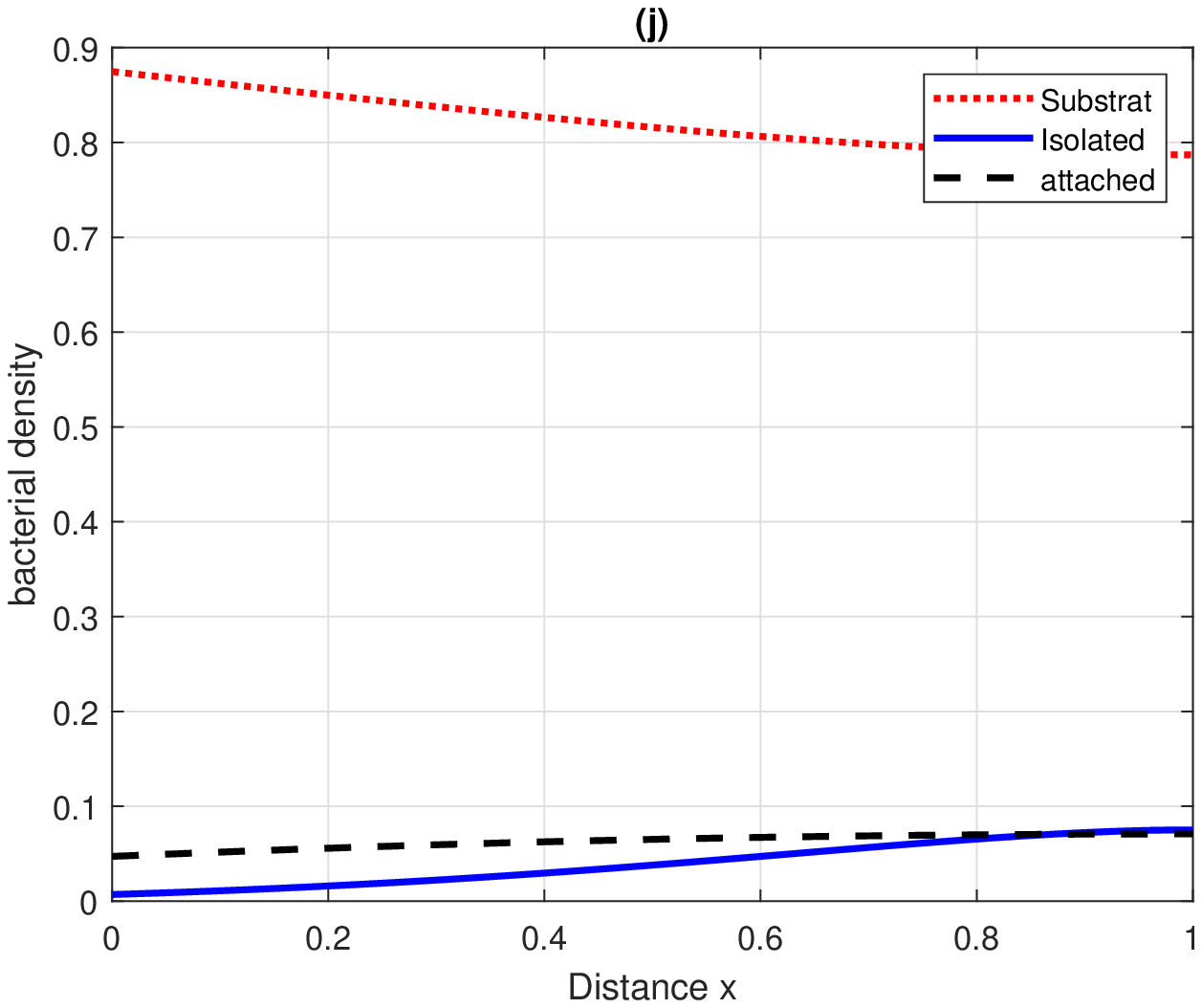}
\includegraphics[scale=0.5]{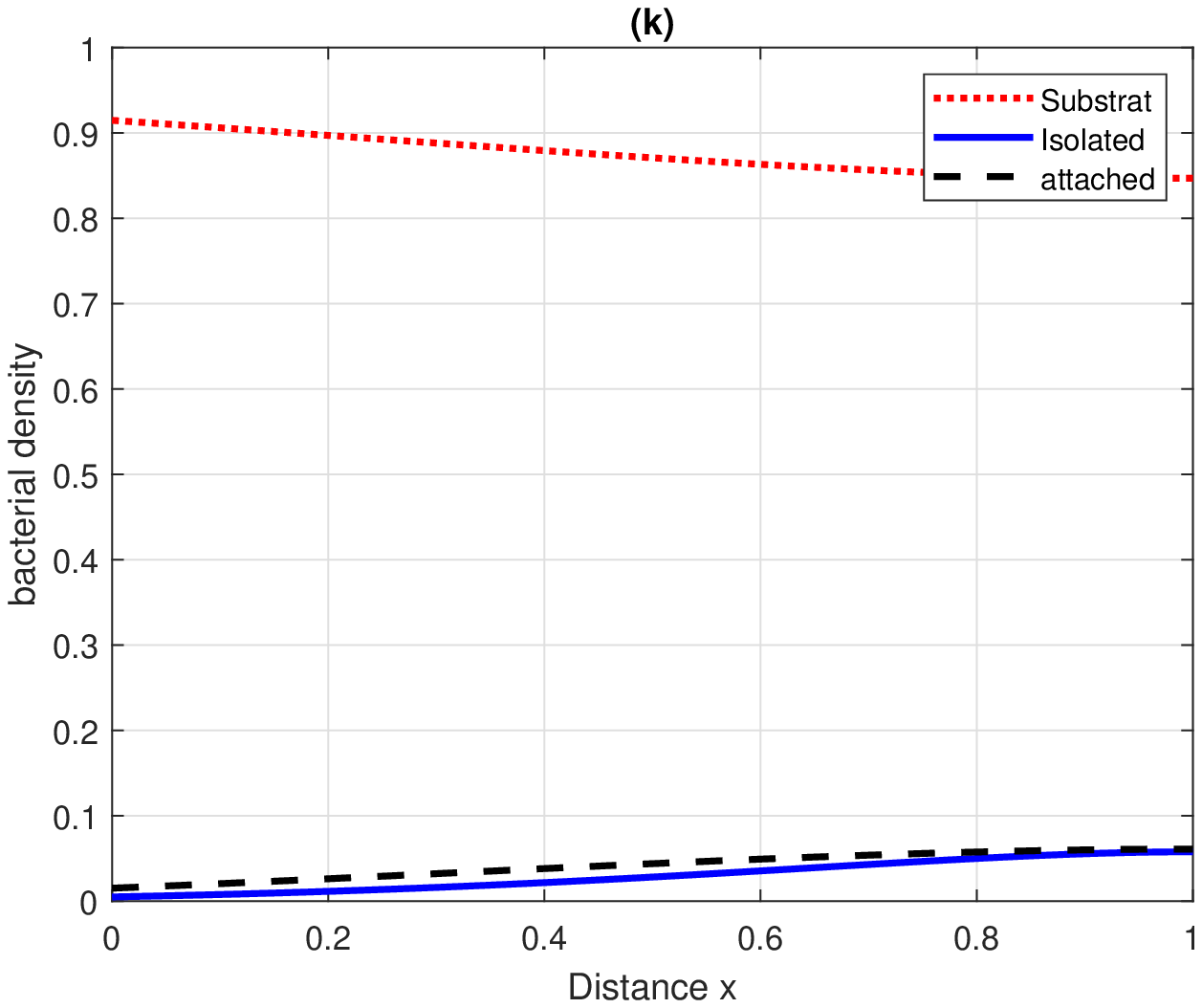}
\includegraphics[scale=0.5]{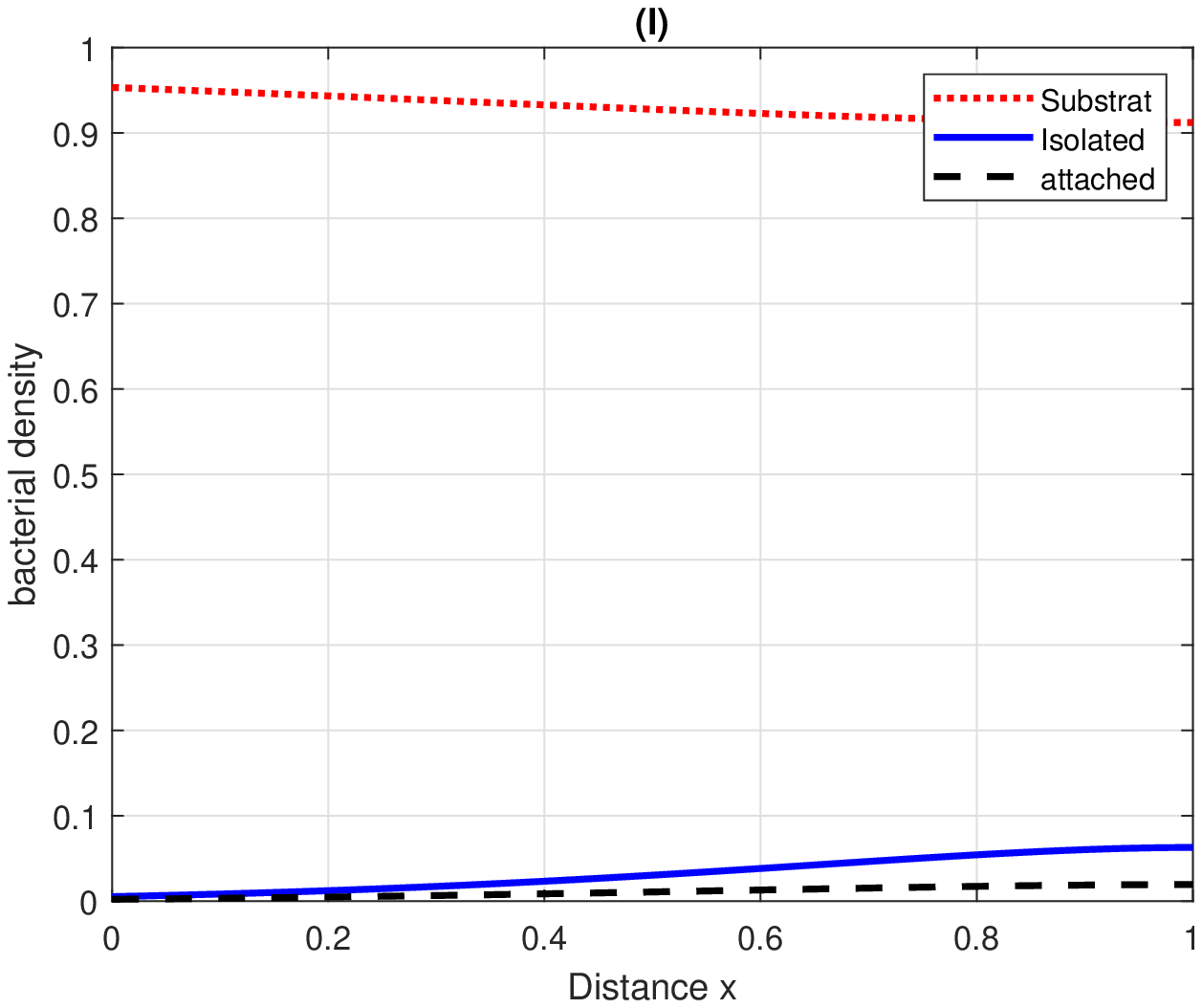}
\includegraphics[scale=0.5]{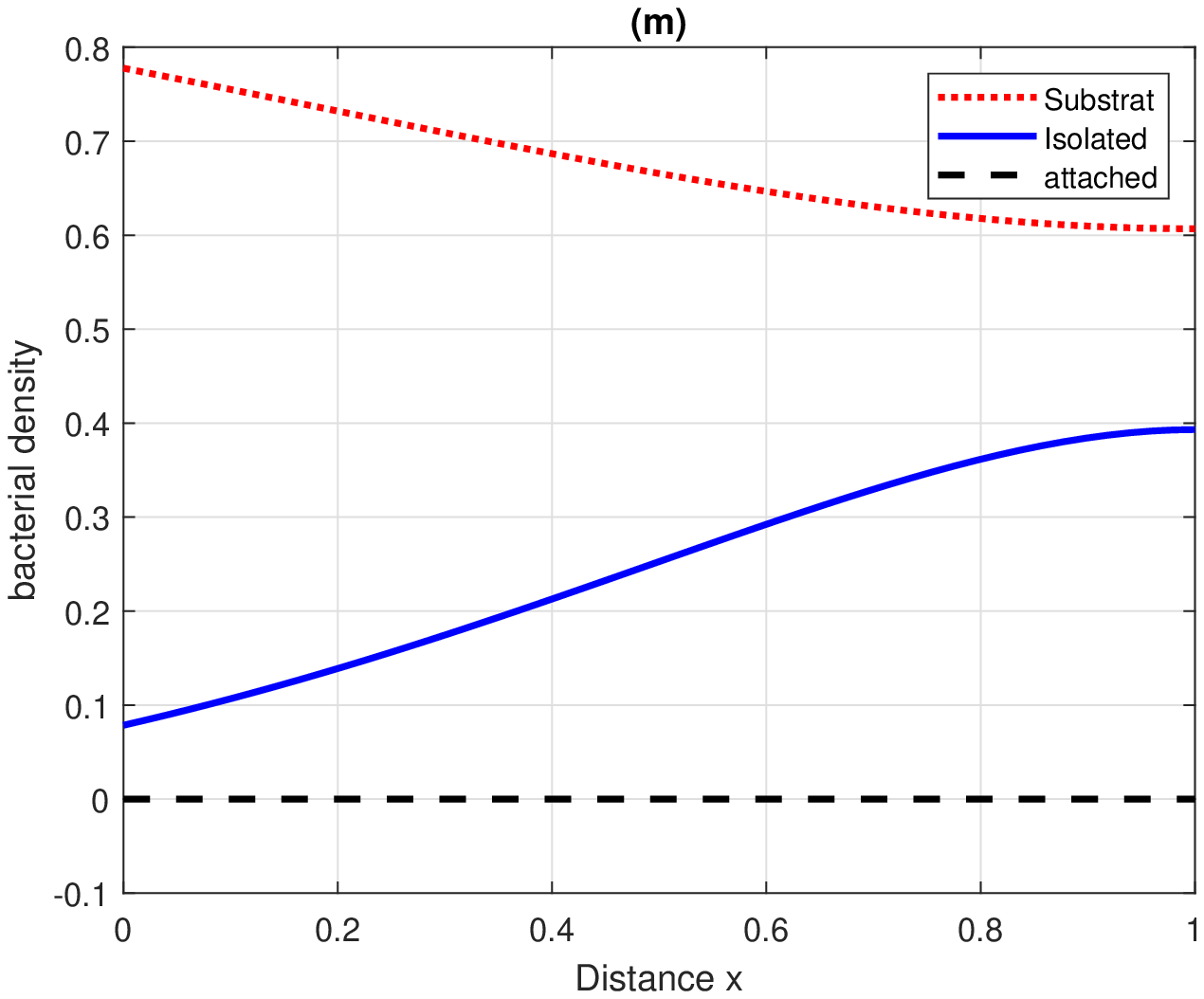}
\caption{The solution $(S, u, v)$ of system (2.1)-(2.3) as a function of two variables $x$ and $t=100$, in the case $f(S)=\displaystyle\frac{4S}{1+S},\,\,g(S)=\displaystyle\frac{5S}{1+S},$ $(h)$ $(d_{0}=1,d_{1}=100,d_{2}=0.1)$, $(i)$ $(d_{0}=1,d_{1}=1,d_{2}=0.1)$, $(j)$ $(d_{0}=1,d_{1}=0.2,d_{2}=1)$, $(k)$ $(d_{0}=1,d_{1}=0.2,d_{2}=0.3)$, $(l)$ $(d_{0}=1,d_{1}=0.2,d_{2}=0.2)$, $(m)$ $(d_{0}=1,d_{1}=0.3,d_{2}=0.2)$.}
\end{figure}
\begin{figure}[!hbtp]

\centering
\includegraphics[scale=0.5]{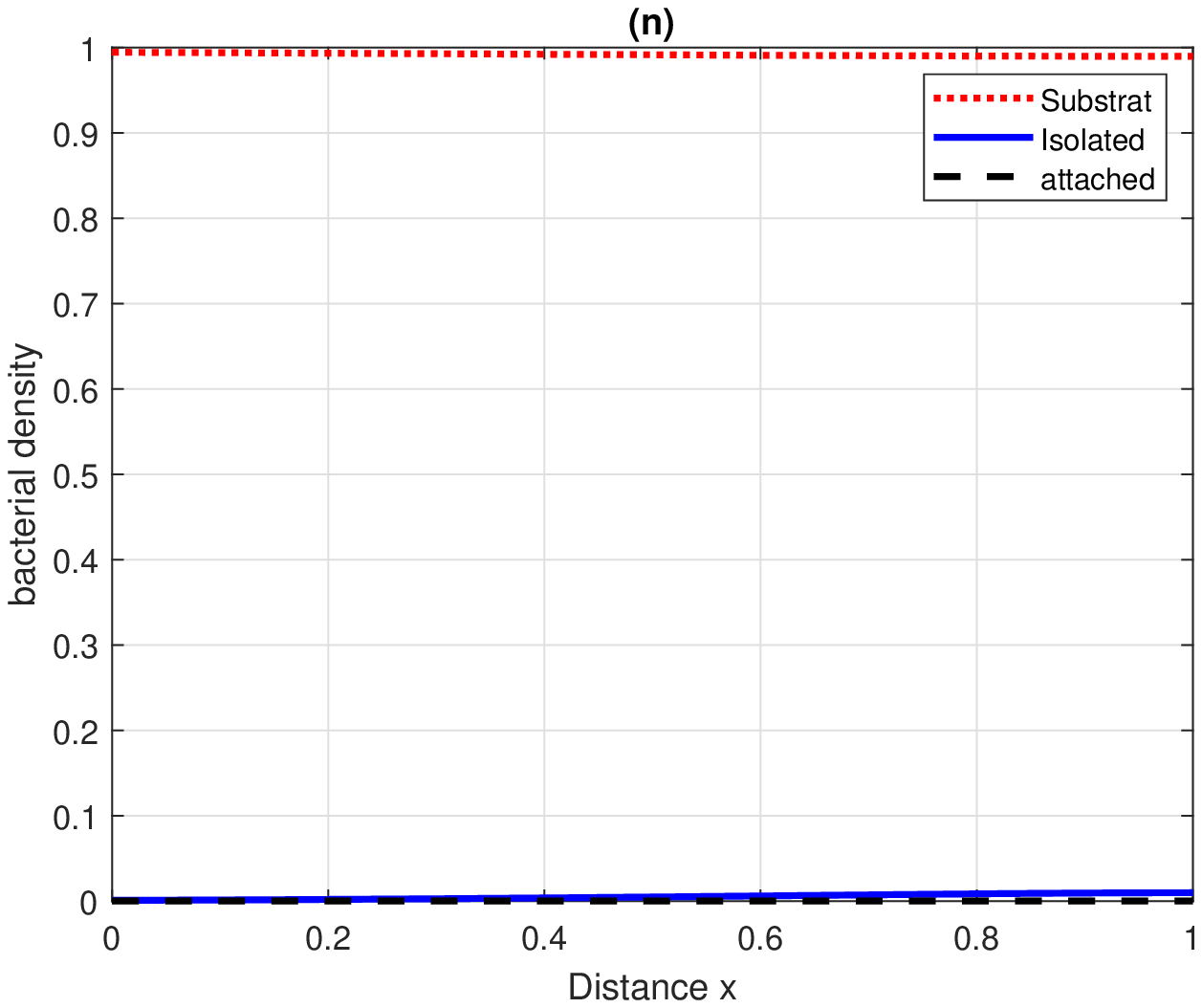}
\includegraphics[scale=0.5]{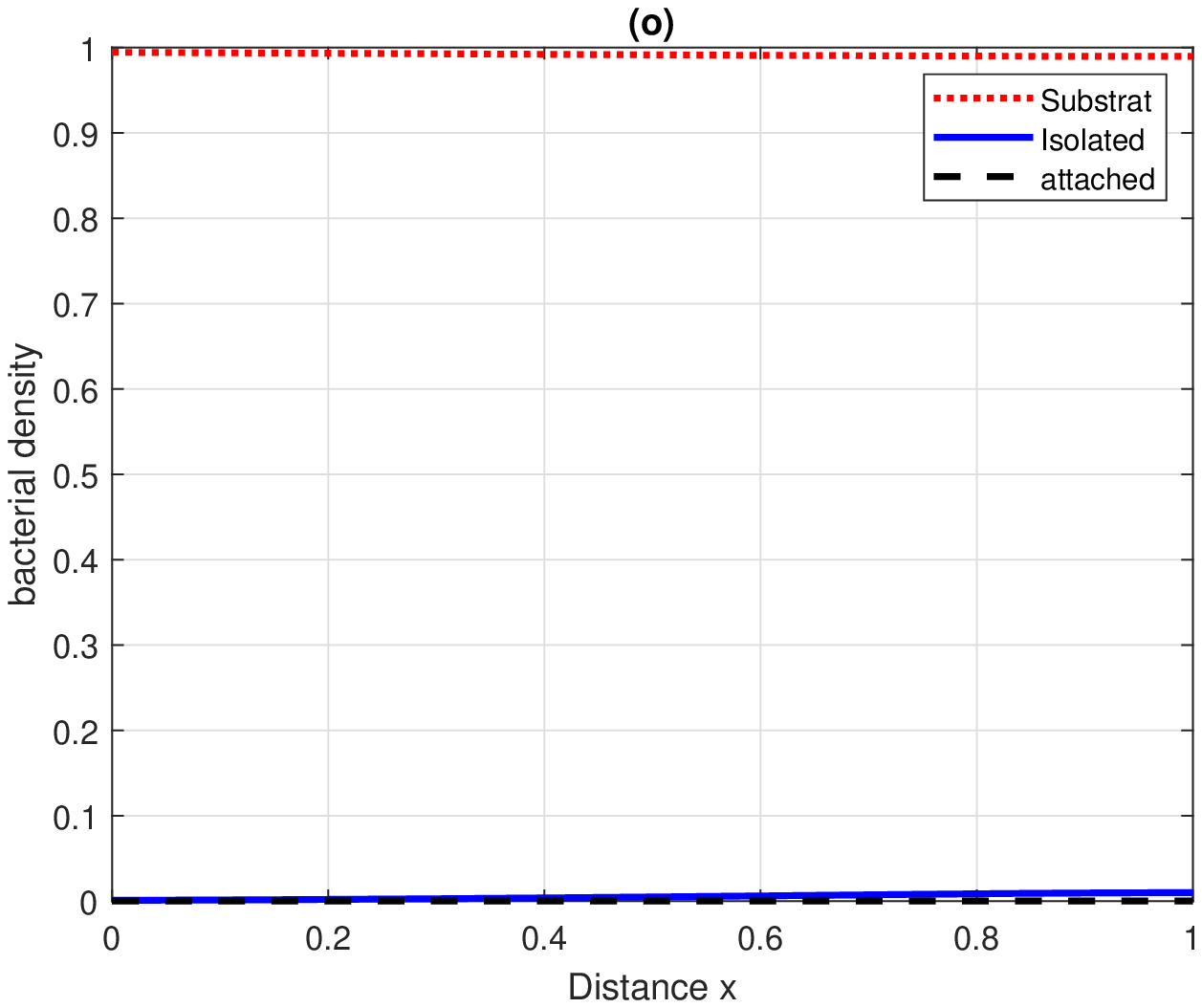}
\includegraphics[scale=0.5]{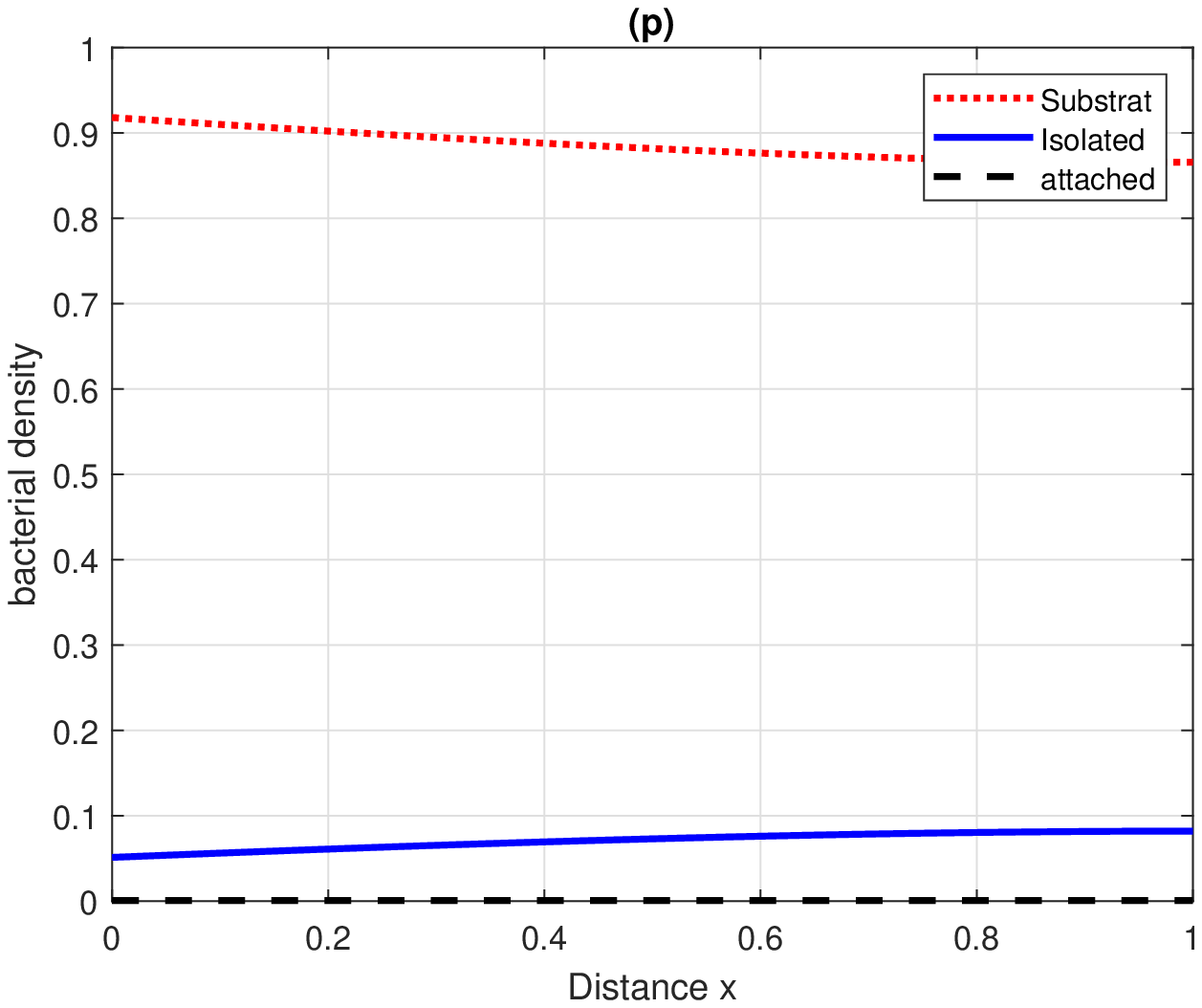}
\caption{The solution $(S, u, v)$ of system (2.1)-(2.3) as a function of two variables $x$ and $t=100$, in the case $f(S)=\displaystyle\frac{4S}{1+S},\,\,g(S)=\displaystyle\frac{5S}{1+S},\,\alpha(u,v)=u+v,\,\beta(u,v)=1$ $(n)$ $(d_{0}=1,d_{1}=0.2,d_{2}=1)$, $(o)$ $(d_{0}=1,d_{1}=0,2,d_{2}=100)$, $(p)$ $(d_{0}=1,d_{1}=1,d_{2}=100)$}
\end{figure}

\FloatBarrier


\begin{thebibliography}{99}
\bibitem{A} H. Amann \textit{Fixed point equations and non-linear eigenvalue problems in ordered Banach spaces}, SIAM Rev. \textbf{18}, (1976), 620.
\bibitem{A1} H. Amann \textit{Dynamic theory of quasilinear parabolic systems. III. Global existence}, Math. Z., 202(2):219–250, 1989.
\bibitem{A2} H. Amann \textit{Nonhomogeneous linear and quasilinear elliptic and parabolic boundary value problems}, in \textit{Function
spaces, differential operators and nonlinear analysis}, (Friedrichroda, 1992), volume 133 of Teubner-Texte Math.,
pages 9–126. Teubner, Stuttgart, 1993.
\bibitem{BJ} M. Ballyk, D. Jones, H.L. Smith \textit{Microbial competition in reactors with wall attachement: a mathematical comparaison of chemostat and plug flow models}, Microbial Ecology. \textbf{41}, No. 3 (2001), 210-221.
\bibitem{BL} M. Ballyk, D. Le, D. Jones, H.L. Smith \textit{Effects of random motility on microbial growth and competition in a flow reactor}, SIAM J. Appl.Math. \textbf{59}, No. 2 (1998), 573-596.
\bibitem{BS} M. Ballyk, H.L. Smith, \textit{A model of microbial growth in a plug flow reactor with wall attachement}, Mathematical Biosciences. \textbf{158}, (1999), 95-126.
\bibitem{FH} R. Fekih-Salem, J. Harmand, C. Lobry, A. Rapaport, T. Sari, \textit{Extensions of the chemostat model with flocculation}, J. Math. Anal. Appl. \textbf{397}, Issue 1, (2013), 292-306.
\bibitem{FH2} R. Fekih-Salem, A. Rapaport, T. Sari, \textit{Emergence of coexistence and limit cycles in the chemostat model with flocculation for a general class of functional responses}, Applied Mathematical Modelling. \textbf{40}, (2016), 7656-7677.
\bibitem{Fel} Klemens Fellner, Jeff Morgan and Bao Quoc Tang, \textit{Uniform-in-time Bounds for Quadratic Reaction-Diffusion Systems with Mass Dissipation in Higher Dimensions}, Discrete and Continuous Dynamical Systems Series S, Vol 14, Number 2, February 2021, 635-651.
\bibitem{FMTY} William~E. Fitzgibbon, Jeffrey~J. Morgan, Bao~Quoc Tang, and Hong-Ming Yin,
\textit{Reaction-diffusion-advection systems with discontinuous diffusion and mass control}, {\em SIAM J. Math Anal.}, 53(6):6771--6803, 2021.
\bibitem{F} R. Freter, \textit{Mechanics that control the microflora in the large intestine}, Human Intestinal Microflora in Health and Disease (D. Hentges, Ed.), Academic Press, New York, (1983).
\bibitem{FB} R. Freter, H. Brickner, S. Temme, \textit{An understanding of colonization resistance of the mammalian large intestine requires mathematical analysis}, Microecology and Therapy. \textbf{16}, (1986), 147-155.
\bibitem{HA} B. Haegeman, A. Rapaport, \textit{How flocculation can explain coexistence in the chemostat}, J. Biol. Dyn. \textbf{2}, (2008), 1-13.
 \bibitem{G} Guo Lin, Wan-Tong Li, Mingju Ma, \textit{Traveling wave solutions in delayed reaction
difusion systems with applications to
multi-species models}, Discrete Cont. Dyn. Sys. \textbf{13}, (2010), 393-414.
\bibitem{KST} Anna Kostianko, Chunyou Sun, Bao Q. Tang, Juan Yang and Sergey Zelik, \textit{Non-concentration phenomenon reaction-diffusion systems with mass dissipation}, arXiv preprint,
\url{https://doi.org/10.48550/arXiv.2205.02498}, (2022).
\bibitem{Kui} Hendrik J. Kuiper, \textit{Invariant sets for nonlinear elliptic and parabolic systems}, SIAM J. Math. Anal., 11 (6), 1980, 1075-1103.
\bibitem{LSU68} Olga~A. Lady{\v{z}}enskaja, Vsevolod~Alekseevich Solonnikov, and Nina~N. Uralceva, {\em Linear and quasi-linear equations of parabolic type}, volume~23, American Mathematical Soc., 1988
\bibitem{MT21} J. Morgan, B.Q. Tang \textit{Global well-posedness for volume-surface reaction-diffusion systems}, Communications in Contemporary Mathematics, https://doi.org/10.1142/S021919972250002X, (2022).
\bibitem{morgan1989global}
Jeff Morgan,
\textit{Global existence for semilinear parabolic systems}, {\em SIAM journal on mathematical analysis}, 20(5):1128--1144, 1989.
\bibitem{Mor90} (MR1062398) [10.1137/0521064]
J. Morgan,
\textit{Boundedness and decay results for reaction-diffusion systems},
\emph{SIAM Journal on Mathematical Analysis}, \textbf{21} (1990), 1172--1189.
\bibitem{Nit}
Robin Nittka,
\textit{Inhomogeneous parabolic neumann problems}, {\em Czechoslovak Mathematical Journal}, 64(3):703--742, 2014.
\bibitem{Pie1} M. Pierre, \textit{Weak solutions and supersolutions in $L^1$ for reaction-diffusion systems}, J. Evol. Equ., 3 (2003), 153-168.
\bibitem{Pie2} M. Pierre, \textit{GLobal existence in reaction-diffusion systems with dissipation of mass: a survey}, Milan J. Math., 78 (2) (2010), 417-455.
\bibitem{Pie3} Michel Pierre and Didier Schmitt, \textit{Examples of finite time blow up in mass dissipative
reaction-diffusion systems with superquadratic growth}, Discrete and Continuous Dynamical Systems, to appear.
\bibitem{Sharma} V. Sharma and Jeff Morgan, \textit{Global Existence of Solutions to Reaction Diffusion Systems with Mass Transport Type Boundary Conditions}, SIAM Journal on Mathematical Analysis 48 (6) 2016, 4202-4240.
\bibitem{SZ} H. Smith, X,Q. Zhao, \textit{Microbial growth in a plug flow reactor with wall adherence and cell motility}, BioSystems. \textbf{145},(2016), 53-66.
\bibitem{STY} Chunyou Sun, Bao Q. Tang and Juan Yang, \textit{Regularity analysis for reaction-diffusion systems with cubic growth rates}, arXiv preprint,
\url{https://doi.org/10.48550/arXiv.2111.14529}, (2022).
\bibitem{T} B. Tang, A. Sittomer, T. Jackson, \textit{Population dynamics and competition in chemostat models with adaptative nutrient uptake}, J. Math. Biol. \textbf{35}, (1997), 453-479.
\bibitem{Tr} E. Trofimchuk, S. Trofimchuk, \textit{Admissible wavefront speeds for a single species reactiondiffusion equation with delay}, Discrete Cont. Dyn. Sys. \textbf{20}, (2008), 407-423.
\bibitem{ZA} S. Zermani, N. Abdellatif, \textit{On a reaction-diffusion system of flocculation type},  J. Math. Anal. Appl. \textbf{506}, (2022), 125484.
\bibitem{ZL} Q. Zhao, M. Yi, Y. Liu, \textit{Spatial distribution and dose-response relationship for different operation modes in a reaction-diffusion model of the MAPK cascade}, Phys. Biol. 8 \textbf{5}, (2011), 055004.
\end{thebibliography}
\end{document}